\documentclass[a4paper]{amsart}
\usepackage[utf8]{inputenc}

\usepackage{amsmath,amssymb,amsthm}
\usepackage{eucal}
\usepackage{bm}
\usepackage{graphicx,color}
\usepackage[british]{babel}
\usepackage{esint}
\usepackage{subcaption}

\newtheorem{theorem}{Theorem}
\newtheorem{lemma}[theorem]{Lemma}
\theoremstyle{definition}
\newtheorem{remark}[theorem]{Remark}
\newtheorem{example}[theorem]{Example}

\numberwithin{equation}{section}
\numberwithin{theorem}{section}

\newcommand{\R}{\mathbb{R}}
\newcommand{\pw}{\mathrm{pw}}
\newcommand{\nc}{\mathrm{NC}}
\newcommand{\osc}{\mathrm{osc}}
\newcommand{\tri}{\mathcal{T}}
\newcommand{\cof}{\operatorname{cof}}
\newcommand{\ennorm}[1]{\lvert\!\lvert\!\lvert #1 \rvert\!\rvert\!\rvert}

\title{A posteriori error estimates for a modified Morley FEM}
\author{A.K.~Dond}
\author{D.~Gallistl}
\author{S.~Nayak}
\author{M.~Schedensack}

\thanks{AKD acknowledges the financial support of ANRF under
        Grant No. ANRF/ARG/009285.}
\thanks{DG was supported by the European Research Council,
        StG \emph{DAFNE}, ID 891734.}

\address[A.K.~Dond, S.~Nayak]{Indian Institute of Science Education and Research,
Thiruvanan\-thapuram, India}
\email{ashadond (at) iisertvm.ac.in}
\email{subhamnayak828 (at) gmail.com}
\address[D.~Gallistl]{Institut f\"ur Mathematik, Universit\"at Jena,
         07743 Jena, Germany}
\email{dietmar.gallistl (at) uni-jena.de}
\address[M.~Schedensack]{Mathematisches Institut, Universit\"at Leipzig,
            PF 10 09 20, 04009 Leipzig, Germany}
\email{mira.schedensack (at) math.uni-leipzig.de}

\date{}%

\begin{document}

\begin{abstract}
Residual-based a~posteriori error estimators are derived for the 
modified Morley FEM, proposed by Wang, Xu, Hu 
[J. Comput. Math, 24(2), 2006],
for the singularly perturbed biharmonic equation and the nonlinear 
von K\'arm\'an equations. The error estimators are proven to be
reliable and efficient. Moreover, an adaptive algorithm 
driven by these error estimators is investigated in numerical
experiments.
\end{abstract}



\maketitle

\section{Introduction}

Conforming finite element methods (FEM) for fourth-order 
problems require globally $\mathcal{C}^1$ continuity across 
element boundaries such that relatively high polynomial 
degrees have to be used. 
The Morley FEM  \cite{Morley1968} overcomes this drawback
by breaking with the conformity of the trial space
and enforcing the continuity in the vertices of the 
underlying triangles and the normal derivatives on the 
faces only. However, these discrete functions are in 
general not only not in the Sobolev space $H^2(\Omega)$,
but also not in the Sobolev space $H^1(\Omega)$.
While it is well-understood that the Morley FEM 
converges for the biharmonic problem \cite{Ciarlet1978,Gudi2010},
the authors of \cite{NilssenTaiWinther2001} show 
that it may diverge for second order problems.

This paper proves a~posteriori error estimates for two 
fourth-order problems with second-order terms.
The first problem is the singularly perturbed 
fourth-order problem
in a bounded, open, polygonal Lipschitz-domain $\Omega\subseteq \R^2$:
Seek $u\in H^2_0(\Omega)$ with 
\begin{align}\label{e:intro_sing_pert_strong_form}
    \varepsilon^2 \Delta^2 u -\Delta u = f 
    \qquad \text{in }\Omega,
\end{align}
where $0<\varepsilon\ll 1$ and $f\in L^2(\Omega)$.
In \cite{WangXuHu2006} (with a generalisation to 3D in 
\cite{WangMeng2007}), the authors define
a modified Morley FEM by employing  
the nodal interpolation operator of the Morley 
finite element function in the second-order term.
As the nodal evaluations are degrees of freedom of 
the Morley finite element functions, this interpolation 
operator is easily computable.
The authors of \cite{WangXuHu2006}
bound the error a~priori by $h^{1/2}$, $h$ being
the maximal mesh size, where the 
reduced convergence rate results from the boundary layer 
of the solution. This motivates adaptive mesh refinement.

Adaptive mesh-refinement is usually driven by a reliable 
and efficient error estimator $\eta$, which means that 
the error estimator is (up to some higher-order terms) 
an upper and lower bound for the error.
In \cite{ZhangWang2008} an a posteriori error estimator 
is derived for the modified Morley FEM for the singularly 
perturbed fourth-order equation, but its efficiency relies 
on the smoothness of the exact solution.
For a mixed FEM on convex domains, \cite{DuLinZhang2022} derive a~posteriori 
error estimates.
In \cite{GallistlTian2024} the authors derive a~posteriori 
error estimates under some abstract assumptions that include 
a wide class of nonconforming FEMs, but the assumptions 
do not apply to the modified Morley FEM.

In the first part of this paper, we derive an a~posteriori error estimator 
$\eta$ and prove that it is reliable and efficient in the sense of 
\begin{align*}
    &\ennorm{u-u_h}_{\varepsilon,\pw} + \lVert \nabla(u-I_h u)\rVert_{L^2(\Omega)}
    \lesssim \eta \\
    &\qquad\qquad\qquad
    \lesssim \ennorm{u-u_h}_{\varepsilon,\pw} 
      + \lVert \nabla(u-I_h u)\rVert_{L^2(\Omega)}
      + \osc_\varepsilon(f,\tri)
\end{align*}
for the singular perturbed
fourth-order equation, where $I_h$ denotes the nodal 
interpolation operator. The precise definition of the energy
norm $\ennorm{\bullet}_{\varepsilon,\pw}$ and the oscillations
can be found in Section~\ref{s:sing}.
The proof relies on techniques developed in 
\cite{Verfuerth1998} for the singularly perturbed 
reaction-diffusion equation.

The second part of this paper is concerned with the 
nonlinear von K\'arm\'an equations, which seek in their strong form 
$\psi_1,\psi_2\in H^2_0(\Omega)$ with 
\begin{align*}
    \Delta^2 \psi_1 &= [\psi_1,\psi_2] + f,\\
    \Delta^2 \psi_2 &= -\frac{1}{2}[\psi_1,\psi_1]
\end{align*}
in the bounded polygonal Lipschitz-domain $\Omega\subseteq \R^2$.
Here, the Monge-Amp\`ere form or von K\'arm\'an bracket $[\bullet,\bullet]$
is defined as 
\begin{align*}
    [\theta,\chi] 
    := \theta_{xx}\chi_{yy} + \theta_{yy}\chi_{xx} - 2 \theta_{xy}\chi_{xy}
    = \cof(D^2\theta):D^2\chi, 
\end{align*}
where $\cof(A) = (A_{22}, -A_{12}; -A_{21}, A_{11})$ denotes the 
cofactor of a two dimensional matrix $A$ and $:$ is the Frobenius 
inner product.
This von K\'arm\'an bracket introduces a nonlinearity in the equations,
which in the weak form can be equivalently included by a term like 
\begin{align}\label{e:intro_b}
    \int_\Omega \cof(D^2 \psi_1):D^2\psi_2 \,\theta\,dx
    = \int_\Omega (\cof(D^2 \psi_1)\nabla\psi_2)\cdot \nabla\theta\,dx
\end{align}
for all test functions $\theta\in H^2_0(\Omega)$.
However, this equality is no longer true for nonconforming Morley functions
and \cite{MallikNataraj2016} employs the second version in a 
Morley FEM for the von K\'arm\'an equations, but without reliable and efficient 
a~posteriori error estimates.
To overcome the difficulties in the a~posteriori error estimates,
\cite{CarstensenMallikNataraj2020} uses the first integral in 
\eqref{e:intro_b} for a nonconforming Morley FEM, such that the 
nonlinearity can be treated similar to an $L^2$ term instead 
of a Laplace-type term. However, in an optimal control problem 
in \cite{ChowdhuryDondNatarajShylaja2022}, the missing 
symmetry between the second and the third variable in the first 
integral in \eqref{e:intro_b} caused some problems in the 
a~posteriori error analysis for the adjoint equation.

The second part of this paper uses the second version in \eqref{e:intro_b}
and adopts the ideas of the first part,
namely to include the nodal interpolation operator in this integral
(the nonlinearity) to derive a reliable and efficient error estimator.

\medskip 
The remaining parts of this article are organised as follows:
Section~\ref{s:prelim} defines the discrete spaces and 
operators. Section~\ref{s:sing} defines the singularly perturbed
biharmonic problem, its discretization and the error estimator 
for this problem and proves the reliability and efficiency 
of this error estimator, while Section~\ref{s:vK} 
is devoted to the von K\'arm\'an equations and its reliable 
and efficient error estimator. 
Section~\ref{s:num} concludes the paper with numerical 
experiments.

\medskip
Standard notation on Lebesgue and Sobolev spaces
applies throughout this paper. The $L^2$ inner product is denoted
by
$(v,w)_{L^2(\omega)}$ for a domain $\omega\subseteq \Omega$
and $\lVert\bullet\rVert_\omega:=\lVert\bullet\rVert_{L^2(\omega)}$,
while the $H^k$ norm over $\omega\subseteq \Omega$
is denoted by
$\lVert\bullet\rVert_{H^k(\omega)}$.
The notation $A\lesssim B$ abbreviates $A\leq C B$ for some
constant $C$ that is independent of the mesh size and
the singular perturbation parameter $\varepsilon>0$.
The notation $A\approx B$ abbreviates $A\lesssim B\lesssim A$.

\section{Preliminaries}\label{s:prelim}

Let $\tri$ be a regular triangulation of the open, bounded, and connected
Lipschitz-polygon $\Omega$ from a shape-regular
family. 
Let $\mathcal{N}$ denote the set of vertices in $\tri$ and 
$\mathcal{N}(\Omega)$ the set of interior vertices.
Let $\mathcal{F}$ denote the set of faces of $\tri$ and
$\mathcal{F}(\Omega)$ the set of interior faces.
For any interior face $F\in\mathcal{F}(\Omega)$ we fix
the two triangles $T_+,T_-\in\tri$ with $F=T_+\cap T_-$,
while for a boundary face, let $T_+\in\tri$ be the unique
triangle with $F\subseteq T_+$.
Then, let $\nu_F=\nu_{T_+}\vert_F$ be the unit normal
on $F$ and $\tau_F = (0, -1; 1, 0) \nu_F$ denotes the unit tangent on $F$.
Furthermore, for an interior face $F\in\mathcal{F}(\Omega)$,
let $[v]_F:=v\vert_{T_+}-v\vert_{T_-}$ denote the jump across $F$
and for a boundary face $F\in\mathcal{F}\setminus\mathcal{F}(\Omega)$,
let $[v]_F:=v\vert_{T_+}$.
For an interior face $F\in\mathcal{F}(\Omega)$,
let $\omega_F:=T_+\cup T_-$, while for a boundary face $F$,
define the face patch by $\omega_F=T_+$.
The set of faces of a triangle $T$ reads $\mathcal F(T)$.
For $T\in\tri$ define the element patch 
$\omega_T:=\bigcup_{F\in\mathcal F(T)}\omega_F$.
The enlarged patch of an element $T$ is given by 
$\Omega_T:=\bigcup \{K\in\tri\,\vert\, K\cap T\neq \emptyset\}$.
The patch of a vertex $z$ of $\tri$ is defined by 
$\omega_z:=\bigcup \{T\in\tri\,\vert\, z\in T\}$.
For a face $F$, the set
$\Omega_F = \cup_{z\in \mathcal N \cap F} \omega_{z}$
is the union of the nodal patches related to the vertices
belonging to the face $F$.
The diameter of $T$ and $F$ is denoted by $h_T$ and $h_F$,
respectively. The piecewise constant mesh-size function
$h_\tri$ is defined by $h_\tri|_T=h_T$.

The broken Sobolev space $H^k(\tri)$ is defined by 
\begin{align*}
    H^k(\tri):=\{v\in L^2(\Omega)\,\vert\,
      \forall T\in\tri:\, v\vert_T\in H^2(T)\}.
\end{align*}

Let $P_k(\tri)$ denote the space
of piecewise polynomials with respect to $\tri$
of degree not larger than $k$
and let $\Pi_k$ denote the $L^2$ projection to $P_k(\tri)$.
Further define
\begin{align*}
	S^k(\tri):=P_k(\tri)\cap H^1(\Omega)
	\qquad \text{and} \qquad
	S^k_0(\tri):=P_k(\tri)\cap H^1_0(\Omega).
\end{align*}

Define the space of Morley finite element functions 
\begin{align*}
    \mathcal{M}(\tri)
    :=\left\{v_h\in P_2(\tri)\;\left\vert\;
        \begin{aligned}
            &v_h\text{ is continuous at }\mathcal{N},\\
            &\forall F\in\mathcal{F}(\Omega):\;
            [\nabla v_h(\operatorname{mid}(F))]_F=0
        \end{aligned}
        \right\}\right.,\\
    \mathcal{M}_0(\tri)
    :=\left\{v_h\in \mathcal{M}(\tri)\;\left\vert\;
        \begin{aligned}
            &\forall z\in\mathcal{N}\setminus\mathcal{N}(\Omega):\; v_h(z)=0,\\
            &\forall F\in\mathcal{F}\setminus \mathcal{F}(\Omega):\;
            \nabla v_h(\operatorname{mid}(F))=0
        \end{aligned}
        \right\}\right. .
\end{align*}

\paragraph{Morley interpolation operator.}
Let $I_\mathcal{M}:H^2_0(\Omega)\to\mathcal{M}_0(\tri)$ denote the 
Morley interpolation operator with the property that 
$(I_\mathcal{M}v)(z)=v(z)$,  
the integral mean property
\begin{align}\label{e:integral_mean_property}
    (D^2_\pw I_\mathcal{M} v)\vert_T 
    = \fint_T D^2 v\,dx 
    \qquad \text{for all }T\in\tri \text{ and all } v\in H^2_0(\Omega),
\end{align}
and the approximation and stability estimates 
\cite{HuShiXu2012}
\begin{equation}\label{e:estimate_Morley_interpolation1}
    \sum_{j=0}^2\lVert h_\tri^{-j}D_\pw^{2-j}(v-I_\mathcal{M}v)\rVert_{L^2(\Omega)}
        \lesssim \lVert D^2 v\rVert_{L^2(\Omega)}.
\end{equation}

\paragraph{Quasi-interpolation operator.}
Furthermore, let $I_\mathrm{qi}:H^2_0(\Omega)\to \mathcal{M}_0(\tri)$
denote a quasi-interpolation that follows a canonical construction
\cite{Oswald1994}:
A function is first mapped by the  $L^2$ projection 
to the piecewise quadratic polynomials. This is concatenated
with averaging the degrees of freedom, which
assigns as nodal value to each interior vertex $z$ the arithmetic
mean of the nodal values at $z$ on each triangle containing
$z$; and as normal derivative at the midpoint of any interior 
the average
of the two normal derivatives at that point from the two
neighbouring elements.
With known interpolation bounds for
averaging operators \cite{Oswald1994,DiPietroErn2012}
it can be proved that it satisfies
\begin{equation}\label{e:quasi_interpolation}
\begin{aligned}
    \sum_{j=0}^2 
    \lVert h_\tri^{-j}D_\pw^{2-j} (v-I_\mathrm{qi} v)\rVert_{L^2(\Omega)}
    &\lesssim \lVert D^2 v\rVert_{L^2(\Omega)},\\
    \lVert h_\tri^{-1}(v-I_\mathrm{qi} v)\rVert_{L^2(\Omega)}
    +\lVert \nabla_\pw (v-I_\mathrm{qi} v)\rVert_{L^2(\Omega)}
    &\lesssim \lVert \nabla v\rVert_{L^2(\Omega)}.
\end{aligned}
\end{equation}
See also \cite[Chapter~10.6]{BrennerScott2008}.

\paragraph{Nodal interpolation operator.}
Furthermore, let $I_h:H^2(\Omega)\cup \mathcal{M}(\tri)\to S^1(\tri)$ 
denote the nodal interpolation 
operator, which is $H^2$ stable and has the
approximation properties \cite{BrennerScott2008}
\begin{align}\label{e:nodal_interpolation_approx}
    \sum_{j=0}^2 \lVert h_\tri^{-j}D^{2-j}(v-I_h v)\rVert_{L^p(T)} 
    \leq C(p) \lVert D^2 v\rVert_{L^p(T)}.
\end{align}

\paragraph{Enriching operator.}
Let $J:\mathcal{M}_0(\tri)\to H^2_0(\Omega)$ denote an enriching operator 
with the properties \cite{BrennerGudiSung2010}
\begin{align}\label{e:enriching_stab_approx}
    \sum_{j=0}^2 \lVert h_\tri^{-j}D^{2-j}_\pw(v_h-Jv_h)\rVert_{L^2(\Omega)}
   \lesssim \lVert D^2_\pw v_h\rVert_{L^2(\Omega)}
\end{align}
and 
\begin{align}\label{e:prop_enriching_nodal}
    (J v_h) (z) = v_h(z)
\end{align}
and with the (local) estimates  
\begin{equation}\label{e:enriching_apost}
\begin{aligned}
    \lVert D_\pw^2(v_h-Jv_h)\rVert_{L^2(T)}
    &\approx \sqrt{\sum_{F\in\mathcal{F}, F\cap T\neq\emptyset} 
        h_F^{-1} \lVert [\nabla_\pw v_h]_F\rVert_{L^2(F)}^2},\\
    h_F^{-1} \lVert [\nabla_\pw v_h]_F\rVert_{L^2(F)}
    &\lesssim \inf_{v\in H^2_0(\Omega)} \lVert D_\pw^2(v_h-v)\rVert_{L^2(\omega_F)}
    ,
\end{aligned}
\end{equation}
cf.~\cite{BrennerGudiSung2010}.

\section{Singularly perturbed biharmonic problem}
\label{s:sing}

This section considers the singularly perturbed biharmonic problem
\eqref{e:intro_sing_pert_strong_form}.
The bilinear form $a_\varepsilon:H^2_0(\Omega)\times H^2_0(\Omega)\to\R$
is defined by 
\begin{align*}
    a_\varepsilon(u,v):=\varepsilon^2 \int_\Omega D^2u:D^2 v\,dx
       + \int_\Omega \nabla u\cdot\nabla v\,dx
    \qquad \text{for all }u,v\in H^2_0(\Omega).
\end{align*}
This bilinear form can be extended to
$a_{\varepsilon,\pw}:H^2(\tri)\times H^2(\tri)\to\R$,
which is defined by the piecewise
application of the differential operators, i.e.,
\begin{align*}
    a_{\varepsilon,\pw}(u,v):=\varepsilon^2 \int_\Omega D_\pw^2u:D_\pw^2v\,dx
       + \int_\Omega \nabla_\pw u\cdot\nabla_\pw v\,dx
    \quad \text{for }u,v\in H^2(\tri).
\end{align*}
For the discretization define the discrete 
bilinear form
$a_{\varepsilon,h}:\mathcal{M}(\tri)\times \mathcal{M}(\tri)\to\R$
that includes the nodal interpolation operator $I_h$ on continuous
$P_1$ elements in the gradient terms
\cite{WangXuHu2006}
by 
\begin{align*}
    a_{\varepsilon,h}(u_h,v_h)
    :=\varepsilon^2 \int_\Omega D_\pw^2u_h:D_\pw^2v_h\,dx
       + \int_\Omega \nabla I_h u_h\cdot\nabla I_h v_h\,dx
\end{align*}
for all $u_h,v_h\in \mathcal{M}(\tri)$.
The corresponding norms are defined as 
\begin{align*}
    \ennorm{v}_\varepsilon:=a_\varepsilon(v,v)^{1/2},
    \qquad 
    \ennorm{v}_{\varepsilon,\pw}:= a_{\varepsilon,\pw}(v,v)^{1/2}
\end{align*}
and the respective local norm on patch-like subdomains
$\omega\subseteq\Omega$ as 
\begin{align*}
    \ennorm{v}_{\varepsilon,\pw,\omega}^2
    :=\varepsilon^2 \int_\omega \lvert D_\pw^2v\rvert^2\,dx
            + \int_\omega \vert \nabla_\pw v\rvert^2\,dx
        \quad \text{for all }v\in H^2(\tri(\omega)),      
\end{align*}
where $\tri(\omega)\subseteq\tri$ denotes the set of triangles covering
$\overline\omega.$

The weak formulation of \eqref{e:intro_sing_pert_strong_form} 
seeks $u\in H^2_0(\Omega)$ with 
\begin{align}\label{e:sing_pert_weak_form}
    a_\varepsilon(u,v)=(f,v)_{L^2(\Omega)} 
    \qquad \text{for all }v\in H^2_0(\Omega),
\end{align}
while the discretization seeks $u_h\in \mathcal{M}_0(\tri)$
with 
\begin{align}\label{e:sing_pert_discrete_problem}
    a_{\varepsilon,h}(u_h,v_h) = (f,I_h v_h)_{L^2(\Omega)}
    \qquad \text{for all }v_h\in \mathcal{M}_0(\tri).
\end{align}
A~priori error estimates for this modified Morley FEM 
can be found in \cite{WangXuHu2006}, while
\cite{ZhangWang2008} defines an a~posteriori error estimator,
which is reliable, but efficiency does only hold up to
higher-order Sobolev norms of the exact solution.

Let $\kappa_T:=\min\{1,h_T/\varepsilon\}$ and 
define the local error estimator contributions
\begin{align*}
    \mu_\nc(T)&:=\sqrt{\sum_{F\in\mathcal{F}(T)}
        (\varepsilon/\kappa_T)\, \lVert [\nabla u_h]_F\rVert_{L^2(F)}^2 },\\
    \mu_{I_h}(T)&:= \lVert\nabla(u_h-I_h u_h)\rVert_{L^2(T)},\\
    \eta_f(T)&:= \lVert h_T\kappa_T f\rVert_{L^2(T)},\\
    \eta_1(T)&:= \sqrt{\sum_{F\in\mathcal{F}(T)\cap \mathcal{F}(\Omega)}
            \varepsilon^{3}\kappa_T
               \lVert [D^2_\pw u_h\,\nu]_F\rVert_{L^2(F)}^2
          },\\
    \eta_2(T)&:=\sqrt{\sum_{F\in\mathcal{F}(T)\cap \mathcal{F}(\Omega)}
            h_T\kappa_T^2
               \lVert [\nabla I_h u_h\cdot\nu]_F\rVert_{L^2(F)}^2
          },\\
    \eta(T)&:= \sqrt{\mu_\nc^2(T)+\mu_{I_h}^2(T)+\eta_f^2(T)
        +\eta_1^2(T) + \eta_2^2(T)}, %+ \mu^2(T)
\end{align*}
and the global error estimator
\begin{align*}
        \eta&:=\sqrt{\sum_{T\in\tri} \eta^2(T)}.
\end{align*}

The following theorem proves the reliability of the error estimator.

\begin{theorem}[reliability]\label{t:sp_reliability}
    The exact solution $u\in H^2_0(\Omega)$ to \eqref{e:sing_pert_weak_form} 
    and the discrete solution $u_h\in \mathcal{M}_0(\tri)$ 
    to \eqref{e:sing_pert_discrete_problem} satisfy
    \begin{align*}
        \ennorm{u-u_h}_{\varepsilon,\pw}^2 
        + \lVert\nabla(u-I_h u_h)\rVert_{L^2(\Omega)}^2
        \lesssim \eta^2.
    \end{align*}
\end{theorem}

The proof of this theorem is preceded by the following lemma
that bounds 
the nonconformity error by $\mu_\nc+\mu_{I_h}$ similar as in 
\eqref{e:enriching_apost}, but for the $\varepsilon$-dependent norm.
The proof proceeds similar as in the proof of 
the reliability estimate of \cite[Lemma~3.5]{GallistlTian2024}. 
However, the assumptions of that Lemma are not 
satisfied for the Morley FEM, and, hence, the definition
of the smoothing operator $\hat{\Pi}_C$ has to be modified
in order to make the estimate independent of the ratio
$h/\varepsilon$. This results in the additional
term $\mu_{I_h}$ in the right-hand side.

\begin{lemma}[a~posteriori estimate for nonconformity error]
    \label{l:sp_nc_estimator}
    Any $u_h\in\mathcal{M}_0(\tri)$ satisfies 
    \begin{align*}
        \inf_{v\in H^2_0(\Omega)} \ennorm{u_h-v}_{\varepsilon,\pw}
        \lesssim \mu_\nc + \mu_{I_h}.
    \end{align*}
\end{lemma}

\begin{proof}
    The first step consists in the localisation of 
    \cite[Lemma~3.4]{GallistlTian2024} that proves 
    \begin{align*}
        \inf_{v\in H^2_0(\Omega)} \ennorm{u_h-v}_{\varepsilon,\pw}^2
        \lesssim \sum_{y\in \mathcal{N}} 
         \mathcal B(y,u_h)
    \end{align*}
   where
   \begin{align*}
     \mathcal B(y,u_h)=
          \min_{v_y\in V(\omega_y)}
          \Bigg( \frac{1}{h_y^2 \kappa_y^2} 
                    \lVert u_h-v_y\rVert_{L^2(\omega_y)}^2
              +\frac{1}{\kappa_y^2} 
                  \lVert\nabla_\pw(u_h-v_y)\rVert_{L^2(\omega_y)}^2\\
             +\varepsilon^2 
                  \lVert D^2_\pw (u_h-v_y)\rVert_{L^2(\omega_y)}^2
            \Bigg),
    \end{align*}
    where $\kappa_y=\min\{1,h_y/\varepsilon\}$ with 
    $h_y=\max\{h_T\,\vert\,T\in\tri,\, T\subseteq \omega_y\}$ and 
    $V(\omega_y):=\{v\vert_{\omega_y}\;\vert\; v\in H^2_0(\Omega)\}$.
    We now fix an arbitrary vertex $y$.
    Let $\hat{\tri}$ be a locally uniformly refined triangulation of $\omega_y$ with 
    (maximal) mesh-size $\hat{h}\approx \min\{\varepsilon,h_y\}$
    (with no refinement and thus $\hat h = h_y$ in the case $h_y\leq\varepsilon$).
    Let $\hat{V}_C\subseteq V(\omega_y)$ denote the conforming 
    Hsieh--Clough--Tocher (HCT)
    finite element space \cite{Ciarlet1978} with respect to the triangulation 
    $\hat{\tri}$. On $\partial\Omega\cap\partial\omega_y$,
    the functions of $\hat{V}_C$ have clamped boundary conditions,
    whereas on the remaining part of $\partial\omega_y$, no boundary
    condition is prescribed.
    Let $\mathcal{A}$ denote some averaging operator that
    averages over all adjacent triangles that are contained
    in the patch $\omega_y$.
    Define a smoothing operator
    $\hat{\Pi}_C:\mathcal{M}_0(\tri)\to \hat{V}_C$ by 
    \begin{align*}
        (\hat{\Pi}_C v_h)(z)
             =  (I_h v_h)(z)
    \quad\text{and}\quad
        (\nabla\hat{\Pi}_C v_h)(z)&=\mathcal{A}(\nabla_\pw v_h)(z)
     \end{align*}
   for all interior (with respect to $\Omega$)
   vertices $z$ of $\hat\tri$,
     as well as 
     \begin{align*}
     \fint_{F}\partial_{\nu_F} (\hat{\Pi}_C v_h)\,ds
              = 
      \fint_{F} \mathcal A(\partial_{\pw,\nu_F} v_h)\,ds
     \end{align*}
    for all interior (with respect to $\Omega$)
    faces $F$ of $\hat\tri$
    with a fixed normal $\nu_F$.
    Note that the HCT degrees of freedom are defind as averages 
    because the quantities $\fint_{F} \partial_{\pw,\nu_F} v_h\,ds$
    are possibly multi-valued on (fine) edges $F$ of $\hat T$.

    \textbf{Case 1/2.}
    We begin with the case that $\hat\tri$ coincides with $\mathcal T$
    on $\omega_y$, which means that $h_y\lesssim \hat h$ and thus
    $h_y\lesssim \varepsilon$ so that $\kappa_y\approx h_y/\varepsilon$.
    The conformity $\hat V_C\subseteq V(\omega_y)$
    of the HCT finite element space
    and inverse estimates on the scale $\hat h$ (depending on $y$) show
    \begin{align*}
        \mathcal B(y,u_h)\lesssim
           \frac{\varepsilon^2}{h_y^4} 
                    \lVert u_h-\hat{\Pi}_C u_h\rVert_{L^2(\omega_y)}^2
          .
    \end{align*}
    Further, the nodal values satisfy $\hat{\Pi}_C u_h(z)=I_h u_h(z)=u_h(z)$.
    In this case, standard estimates
    \cite[Proposition~2.5]{Gallistl2015}
    bound $\mathcal B(y,u_h)$ by the sum of all $\mu_\nc(T)^2$
    for $T$ belonging to $\omega_y$.

    \textbf{Case 2/2.}
    In the remaining case that $\hat h<h_y$ and thus
    $\varepsilon\lesssim h_y$ and $\kappa_y\approx 1$,
    we split the difference
    $$
    u_h-\hat{\Pi}_C u_h
    =
    (u_h-I_h u_h)
    + (I_h u_h - \hat{\Pi}_C u_h).
    $$
    The function $u_h-I_hu_h$ is piecewise polynomial with 
    respect to  $\mathcal T$.
    Therefore, the triangle inequality, 
    inverse inequalities, the discrete inequality
    \begin{align*}
        \lVert u_h-I_h u_h\rVert_{L^2(T)}
        \lesssim h_T \lVert \nabla(u_h-I_h u_h)\rVert_{L^2(T)},
    \end{align*}
    for the norms inside $\mathcal B(y,u_h)$
    and $\varepsilon\approx \hat{h}$
    reveal
    $$
    \mathcal B(y,u_h)
    \lesssim
            \frac{1}{\varepsilon^2} 
                        \lVert I_h u_h-\hat{\Pi}_C u_h\rVert_{L^2(\omega_y)}^2
    +
    \|\nabla_\pw (u_h-I_h u_h) \|_{L^2(\omega_y)}^2.
    $$
    
    On every triangle $\hat T$ of $\hat\tri$,
    the degrees of freedom of the HCT element are given by
    $L^{(0)}_{\hat z} v = v(\hat z)$
    and  
    $L^{(1,j)}_{\hat z} v = \partial_j v(\hat z)$
    for the vertices $\hat z$ of $\hat T$ and $j\in\{1,2\}$,
    and 
    $L^{(1,\nu)}_{\hat F} v = \fint_{\hat F} \partial_{\nu_{\hat F}}v\,ds$
    for the faces $\hat F$ of $\hat T$.
    We do not explicitly include the dependence on $\hat T$
    in the notation because it will be clear from the context.
    The local basis functions dual to these degrees of freedoms 
    are denoted by $\hat\varphi^{(0)}_{\hat z}$,
    $\hat \varphi^{(1,j)}_{\hat z}$,
    $\hat \varphi^{(1,\nu)}_{\hat F}$.

    For any vertex $\hat z$ of $\hat T$,
    the definition of $\hat{\Pi}_C$ implies 
    $(I_h u_h-\hat{\Pi}_C u_h)(\hat z)=L^{(0)}_{\hat z}
      (I_h u_h-\hat{\Pi}_C u_h)
     =0$.
    Therefore, the difference $(I_h u_h-\hat{\Pi}_C u_h)$
    can be expanded as follows
    \begin{align}\label{e:IhPih_split}
     \begin{aligned}
      &(I_h u_h-\hat{\Pi}_C u_h)|_{\hat T}
       =
        \sum_{(d,p )}
         L^{(1,d)}_p  (I_h u_h-\hat{\Pi}_C u_h) \varphi^{(1,d)}_p
    \\
       &=
        \sum_{(d,p )}
         L^{(1,d)}_p  (I_h u_h- u_h) \varphi^{(1,d)}_p
   + \sum_{(d,p )}
         L^{(1,d)}_p  ( u_h-\hat{\Pi}_C u_h) \varphi^{(1,d)}_p,
    \end{aligned}
   \end{align}
    where the sums run over the nine possible pairs $(d,p)$
    of type $(j,\hat z)$ and $(\nu,\hat F)$ for 
    $j=1,2$, the vertices $\hat z$, and the faces $\hat F$
    of $\hat T$.

    The basis functions
    associated to the derivative degrees of freedom scale
    in the $L^2$ norm as follows
     \begin{align*}
        \lVert\varphi^{(1,d)}_p\rVert_{\hat{T}}\approx \hat{h}^2.
    \end{align*}
    A direct scaling argument therefore shows for the squared norm of 
    the first sum of the right-hand side of 
    \eqref{e:IhPih_split} that 
    $$
     \left\lVert
         \sum_{(d,p )}
         L^{(1,d)}_p  (I_h u_h- u_h) \varphi^{(1,d)}_p
     \right\rVert_{L^2(\hat{T})}^2
     \lesssim 
      \hat h^2 \lVert \nabla(u_h-I_h u_h)\rVert_{L^2(\hat{T})}^2.
    $$
    Moreover, for the squared norm of the second term
    on the right-hand side of \eqref{e:IhPih_split}
    we compute with the definition
    of $\hat\Pi_Cu_h$ and the scaling of the basis functions 
    \begin{align*}
     &   \left\lVert\sum_{(d,p)}
            L^{(1,d)}_p (u_h -\hat{\Pi}_Cu_h)
               \varphi_p^{(1,d)}\right\rVert_{L^2(\hat{T})}^2\\
    & \lesssim \hat{h}^4
       \Bigg(
          \sum_{j=1}^2\sum_{\hat z\in\mathcal N(\hat T )} 
            \lvert (\partial_j u_h
              - \mathcal{A}(\partial_j u_h))(\hat{z})\rvert^2
       + \sum_{\hat F\in\mathcal F(\hat T)}
          \left\lvert\fint_{\hat F} (\partial_{\nu_{\hat F}}u_h
                - \mathcal A(\partial_{\pw,\nu_{\hat F}}u_h))\,ds\right\rvert^2
       \Bigg)
    \end{align*}
 where $\mathcal N(\hat T)$ and $\mathcal{F}(\hat T)$ are the sets
 of vertices and faces of $\hat T$, respectively.
 With standard arguments \cite{BrennerScott2008}, the differences
 to the average can be bounded by the jumps across all faces of
 $\hat\tri$ having one of their endpoints in $\hat T$,
 namely
    \begin{align*}
         \left\lVert\sum_{(d,p)}
            L^{(1,d)}_p (u_h -\hat{\Pi}_Cu_h)
               \varphi_p^{(1,d)}\right\rVert_{L^2(\hat{T})}^2
     \lesssim 
       \sum_{\substack{\hat F\in\hat{\mathcal{F}},
                        \hat F\cap \hat T\neq\emptyset, \\
                        \hat{F}\not\subseteq
                         (\overline{\partial\omega_y\setminus\partial\Omega})}}
             \hat{h}^3 \lVert [\nabla u_h]_{\hat{F}}\rVert_{L^2(\hat{F})}^2
    \end{align*}
    where $\hat{\mathcal{F}}$ denotes the set of faces in $\hat{\tri}$.
    Note that faces on $\partial\omega_y$ that are not on $\partial\Omega$
    are excluded in the sum because the averaging takes only values from
    inside of $\omega_y$ into account.
Combining the foregoing estimates and summing over all elements
of $\hat\tri$, we obtain
with the bounded overlap of such sets from the shape-regularity
\begin{align*}
 &           \frac{1}{\varepsilon^2} 
                        \lVert I_h u_h-\hat{\Pi}_C u_h\rVert_{L^2(\omega_y)}^2
\\
&  \lesssim
            \frac{1}{\varepsilon^2} 
       \left(
      \hat h^2 \lVert \nabla(u_h-I_h u_h)\rVert_{L^2(\omega_y)}^2
 +       \sum_{\hat F\in\hat{\mathcal{F}},
                \hat{F}\not\subseteq
                         (\overline{\partial\omega_y\setminus\partial\Omega})}
             \hat{h}^3 \lVert [\nabla u_h]_{\hat{F}}\rVert_{L^2(\hat{F})}^2
    \right).
\end{align*}
Note that $[\nabla u_h]_{\hat{F}}=0$, if 
$\hat{F}\not\subseteq F$ for all coarse faces $F\in\mathcal{F}$, 
and therefore
\begin{align*}
     \sum_{\hat{F}\in\hat{\mathcal{F}},
                \hat{F}\not\subseteq
                         (\overline{\partial\omega_y\setminus\partial\Omega})}
          \lVert [\nabla u_h]_{\hat{F}}\rVert_{L^2(\hat{F})}^2
     =\sum_{F\in \mathcal{F},y\in F}
          \lVert [\nabla u_h]_{F}\rVert_{L^2(F)}^2.
\end{align*}
Since in the present case $\hat h\approx \varepsilon$, the result follows. 
\end{proof}

\begin{remark}
In the definition of $\hat \Pi_C$ in the foregoing proof,
the nodal values of $I_h u_h$ are taken instead of 
simple averages of $u_h$. The latter choice would not lead
to an efficient bound with the proof techique employed 
here. This has been commented on in
\cite[Remark 3.8]{GallistlTian2024}, where the averaging process
was analyzed and the authors concluded efficiency for 
the case of continuous trial functions.
In the present case of the Morley element with discontinuous
trial functions, the nodal values are taken from the continuous
object $I_hu_h$, which results in an efficient bound.
\end{remark}

\begin{remark}
    In \cite{WangMeng2007}, the authors introduce a generalisation
    of the modified Morley FEM for the singularly perturbed biharmonic
    problem to 3d. The nodal interpolation operator in the
    $L^2$ term has to be replaced by the interpolation in a space
    of functions whose integral mean over edges are continuous, but
    that could be discontinuous in general.
    The analysis of Lemma~\ref{l:sp_nc_estimator} relies
    on the fact that the degrees of freedom associated with
    $L^2$ terms vanish in \eqref{e:IhPih_split}.
    This would not longer be the case in 3d, and therefore
    the analysis in 3d requires a different approach.
\end{remark}

\begin{proof}[Proof of Theorem~\ref{t:sp_reliability}]
    \textbf{Step 1 (Error split).}
    The triangle inequality leads to
    \begin{align*}
        \lVert \nabla(u-I_h u_h)\rVert_{L^2(\Omega)}
        &\leq \lVert \nabla_\pw(u- u_h)\rVert_{L^2(\Omega)}
           + \lVert \nabla_\pw(u_h-I_h u_h)\rVert_{L^2(\Omega)}\\
        &\leq \ennorm{u-u_h}_{\varepsilon,\pw} + \mu_{I_h}.
    \end{align*}
    The error in the energy norm can be split as 
    \begin{align*}
        &\ennorm{u-u_h}_{\varepsilon,\pw}^2\\
        &\qquad
        =\inf_{v\in H^2_0(\Omega)} \ennorm{u_h-v}_{\varepsilon,\pw}^2
           + \sup_{v\in H^2_0(\Omega),\ennorm{v}_\varepsilon=1}
            \left( 
             (f,v)_{L^2(\Omega)} - a_{\varepsilon,\pw}(u_h,v) \right)^2.
    \end{align*}
    The first term is bounded with the help of 
    Lemma~\ref{l:sp_nc_estimator} by $\mu_\nc^2+\mu_{I_h}^2$.
    For the analysis of the second term,
    let $v\in H^2_0(\Omega)$ with $\ennorm{v}_\varepsilon=1$ and 
    set $e_v:=v-I_\mathrm{qi}v$ for the Morley quasi interpolation operator
    $I_\mathrm{qi}:H^2_0(\Omega)\to\mathcal{M}_0(\tri)$
    from \eqref{e:quasi_interpolation}.
    The expression inside the squared supremum then reads
    \begin{equation*}
    \begin{aligned}
        &(f,v)_{L^2(\Omega)} - a_{\varepsilon,\pw}(u_h,v)
        = \sum_{j=1}^4 T_j
    \end{aligned}
    \end{equation*}
    with 
    \begin{align*}
        T_1&:=(f,e_v)_{L^2(\Omega)} + (f, I_\mathrm{qi}v-I_hI_\mathrm{qi}v)_{L^2(\Omega)},
        &
        T_2&:=-\varepsilon^2 (D^2_\pw u_h,D^2_\pw e_v)_{L^2(\Omega)},\\
        T_3&:=- (\nabla I_h u_h,\nabla (v-I_h I_\mathrm{qi}v))_{L^2(\Omega)},
        &
        T_4&:= (\nabla (I_h u_h-u_h),\nabla v)_{L^2(\Omega)}.
    \end{align*}
    It remains to bound these terms by the error estimator.

    \textbf{Step~2 (Bound of $T_1$).}
    The estimates 
    \eqref{e:quasi_interpolation}
    of the nonconforming Morley quasi interpolation operator,
    the Cauchy inequality, and $\ennorm{v}_{\varepsilon}=1$ 
    prove
    \begin{align*}
        (f,v-I_\mathrm{qi}v)_{L^2(\Omega)} 
        &\lesssim 
          \min\{\lVert h_\tri f \rVert_{L^2(\Omega)} \,
            \lVert \nabla v \rVert_{L^2(\Omega)},
             \lVert h_\tri^2 f \rVert_{L^2(\Omega)} \,
            \lVert D^2 v \rVert_{L^2(\Omega)}\}\\
        &\leq \lVert\min\{h_\tri,h_\tri^2/\varepsilon\}
           f \rVert_{L^2(\Omega)}
         \leq \eta.
    \end{align*}
    Similarly, the $H^2$ approximation property of $I_h$ and
    the inverse estimate combined with the $H^1$ and $H^2$ stability of 
    $I_{qi}$ lead to
    $
        (f, I_\mathrm{qi}v-I_hI_\mathrm{qi}v)_{L^2(\Omega)}
        \lesssim \eta.
    $
    
    \textbf{Step~3 (Bound of $T_2$).}
    Since the jump and the average of 
    $D^2_\pw u_h\,\nu_F$ over $F$  is constant for any $F\in\mathcal{F}$
    and $[\nabla_\pw e_v]_F$ is affine on $F$ and vanishes in the midpoint of $F$,
    a piecewise integration by parts for $T_2$ leads to
    \begin{align}\label{e:proof_reliability_2}
        |T_2|
        =| \varepsilon^2 (D^2_\pw u_h,D^2_\pw e_v)_{L^2(\Omega)}|
        = \left |\varepsilon^2\sum_{F\in\mathcal{F}(\Omega)}
            ([D^2_\pw u_h\,\nu]_F,\nabla_\pw e_v)_{L^2(F)}
          \right|
          .
    \end{align}
    The multiplicative trace inequality proves for some
    $T_F\in\tri$ with $F\subseteq T_F$ that
    \begin{align}\label{e:multitrace}
        \|\nabla_\pw e_v\|_{L^2(F)}
        &\lesssim h_F^{-1/2} \|\nabla_\pw e_v\|_{L^2(T_F)}
            + \|\nabla_\pw e_v\|_{L^2(T_F)}^{1/2}
                \|D^2_\pw e_v\|_{L^2(T_F)}^{1/2}.
    \end{align}
    The approximation and stability properties from
    \eqref{e:quasi_interpolation} then prove
    \begin{align*}
    \|\nabla_\pw e_v\|_{L^2(F)}
        \lesssim \sqrt{\|\nabla  v\|_{L^2(\Omega_{T_F})} 
                    \,\|D^2  v\|_{L^2(\Omega_{T_F})}}
        \lesssim \varepsilon^{-1/2} \ennorm{v}_{\varepsilon,\Omega_{T_F}}.
    \end{align*}
    On the other hand, \eqref{e:multitrace}
    together with the approximation and stability properties
    from \eqref{e:quasi_interpolation} and Young's inequality 
    leads to
    \begin{align*}
        \|\nabla_\pw e_v\|_{L^2(F)}
        &\lesssim h_F^{-1/2} \lVert \nabla_\pw e_v\rVert_{L^2(T_F)}
           + h_F^{1/2} \lVert D^2_\pw e_v\rVert_{L^2(T_F)}\\
        &\lesssim h_F^{1/2} \lVert D^2 v\rVert_{L^2(\Omega_{T_F})}
        \lesssim (h_F^{1/2}/\varepsilon)\, \ennorm{v}_{\varepsilon,\Omega_{T_F}}.
    \end{align*}
    The combination of the foregoing displayed formulae leads to
    \begin{align}\label{e:proof_reliability_3}
        \|\nabla_\pw e_v\|_{L^2(F)}
        \lesssim \min\{\varepsilon^{-1/2},h_F^{1/2}/\varepsilon\}
        \ennorm{v}_{\varepsilon,\Omega_{T_F}}.
    \end{align}
    This and \eqref{e:proof_reliability_2} bound $T_2$ as
    \begin{align*}
        |T_2|
        \lesssim \sqrt{\sum_{F\in\mathcal{F}(\Omega)}
            \min\{\varepsilon^{3/2},h_F^{1/2}\varepsilon\}^2
               \lVert [D^2_\pw u_h\,\nu]_F\rVert_{L^2(F)}^2}.
    \end{align*}
    
    \textbf{Step~4 (Bound of $T_3$).}
    Since $\nabla I_h u_h$ is piecewise constant, a piecewise integration
    by parts and $v-I_hI_\mathrm{qi}v\in H^1_0(\Omega)$ 
    imply for $T_3$ that
    \begin{align*}
        -T_3=
        (\nabla I_h u_h,\nabla (v-I_h I_\mathrm{qi} v))_{L^2(\Omega)}
        = \sum_{F\in\mathcal{F}(\Omega)}
          ([\nabla I_h u_h\cdot\nu_F]_F, v-I_h I_\mathrm{qi} v)_{L^2(F)}.
    \end{align*}
    Note that the $H^2$ stability of the nodal interpolation 
    operator $I_h$ from \eqref{e:nodal_interpolation_approx},
    an inverse inequality, and the $H^1$ stability of $I_\mathrm{qi}$ 
    imply for any $T\in\tri$ that 
    \begin{align*}
        h_T^{-1}\lVert I_\mathrm{qi} v-I_h I_\mathrm{qi}v\rVert_{L^2(T)} 
        &+ \lVert \nabla(I_\mathrm{qi} v-I_h I_\mathrm{qi}v)\rVert_{L^2(T)}
        \lesssim h_T \lVert D^2 I_\mathrm{qi} v\rVert_{L^2(T)} \\
        &\qquad\qquad\qquad \qquad\qquad
        \lesssim \lVert \nabla I_\mathrm{qi} v\rVert_{L^2(T)} 
        \lesssim \lVert \nabla  v\rVert_{L^2(\Omega_T)},
    \end{align*}
    and, therefore, the operator $I_h I_\mathrm{qi}$ enjoys the 
    same approximation and stability properties 
    \eqref{e:quasi_interpolation} as $I_\mathrm{qi}$.
    Hence, the trace inequality and the 
    $H^1$ approximation and stability properties prove 
    \begin{align*}
        \lVert v-I_h I_\mathrm{qi}v\rVert_{L^2(F)}
        \lesssim h_T^{1/2}\lVert \nabla v\rVert_{L^2(\Omega_F)}
        \leq h_T^{1/2} \ennorm{v}_{\varepsilon, \Omega_F},
    \end{align*}
   while the $H^2$ approximation and stability properties show 
    \begin{align*}
        \lVert v-I_h I_\mathrm{qi}v\rVert_{L^2(F)}
        \lesssim h_T^{3/2}\lVert D^2 v\rVert_{L^2(\Omega_F)}
        \leq h_T^{3/2}\varepsilon^{-1} \ennorm{v}_{\varepsilon, \Omega_F}.
    \end{align*}
    The combination of these two estimates leads to 
    \begin{align*}
        \lVert v-I_h I_\mathrm{qi}v\rVert_{L^2(F)}
        \lesssim h_T^{1/2}\kappa_T \ennorm{v}_{\varepsilon, \Omega_F}.
    \end{align*}
    This and the finite overlap of the patches
    eventually show
    \begin{align*}
        |T_3|
        \lesssim  \sqrt{\sum_{F\in\mathcal{F}(\Omega)}
          h_T\kappa_T^2
          \,\lVert [\nabla I_h u_h\cdot\nu_F]_F\rVert_{L^2(F)}^2}
          .
    \end{align*}
    
    \textbf{Step~5 (Bound of $T_4$).}
    Finally, the last term $T_4$
    is estimated with a Cauchy inequality by
    \begin{align*}
       T_4= (\nabla (I_h u_h-u_h),\nabla  v)_{L^2(\Omega)}
        \leq \lVert \nabla (u_h-I_h u_h)\rVert_{L^2(\Omega)} .
    \end{align*}
    This concludes the proof.
\end{proof}

Define the $\varepsilon$-dependent oscillations 
\begin{align*}
    \osc_\varepsilon(f,\tri(\omega_T))
        := \lVert \min\{h_\tri,h_\tri^2/\varepsilon\} (f-\Pi_0 f)\rVert_{L^2(\omega_T)},
\end{align*}
where $\Pi_0$ denotes the $L^2$ projection on the piecewise
constants with respect to $\tri$.
The following theorem proves the efficiency of the error estimator.

\begin{theorem}[efficiency]\label{t:sp_efficiency}
    The error estimator satisfies, for every $T\in\tri$,
    \begin{align*}
        \eta^2(T)
          \lesssim \ennorm{u-u_h}_{\varepsilon,\pw,\omega_T}^2
            +\lVert \nabla (u-I_h u_h)\rVert_{L^2(\omega_T)}^2
            +\osc_\varepsilon(f,\tri(\omega_T)).
    \end{align*}
\end{theorem}

\begin{proof}
    \textbf{Efficiency of $\mu_{I_h}$.}
    A triangle inequality and the definition of the 
    $(\varepsilon,\pw,T)$-norm 
    proves 
    \begin{align*}
    \mu_{I_h}(T)=
        \lVert\nabla(u_h-I_h u_h)\rVert_{L^2(T)}
        &\leq \lVert\nabla(u-I_h u_h)\rVert_{L^2(T)}
           + \lVert\nabla_\pw(u-u_h)\rVert_{L^2(T)}\\
        &\leq \lVert\nabla(u-I_h u_h)\rVert_{L^2(T)}
            + \ennorm{u-u_h}_{\varepsilon,\pw,T}.
    \end{align*}
    
    \textbf{Efficiency of $\eta_f$.}
    The term $\eta_f(T)$ is estimated by the bubble function
    technique \cite{Verfuerth2013}.
    Let $\flat_T=\lambda_1^2\lambda_2^2\lambda_3^2\in H^2_0(T)$ 
    be the $H^2_0$ bubble function, where $\lambda_1,\lambda_2,\lambda_3$
    are the barycentric coordinates of $T$.
    Define $\varphi_T:=\Pi_0 f \flat_T$.
    An equivalence of norms argument shows
    \begin{align*}
        \lVert f\rVert_{L^2(T)}^2
        \lesssim \int_T f \varphi_T\,dx 
           + \lVert f-\Pi_0 f\rVert_{L^2(T)}^2.
    \end{align*}
    Since $u\in H^2_0(\Omega)$ solves \eqref{e:sing_pert_weak_form},
    it satisfies 
    \begin{align*}
        \int_T f \varphi_T\,dx
        = \varepsilon^2 (D^2 u, D^2 \varphi_T)_{L^2(T)} 
          + (\nabla u,\nabla \varphi_T)_{L^2(T)}.
    \end{align*}
    Since $D^2_\pw u_h$ and $\nabla I_h u_h$ are constant on $T$,
    the vanishing boundary conditions of $\varphi_T$ and $\nabla \varphi_T$ 
    imply 
    \begin{align*}
        \varepsilon^2 (D^2_\pw u_h, D^2 \varphi_T)_{L^2(T)} 
          + (\nabla I_h u_h,\nabla \varphi_T)_{L^2(T)}=0.
    \end{align*}
    The combination of the two previously displayed inequalities 
    leads to 
    \begin{align*}
        \int_T f \varphi_T\,dx
        = \varepsilon^2 (D^2_\pw (u-u_h), D^2 \varphi_T)_{L^2(T)} 
          + (\nabla (u-I_h u_h),\nabla \varphi_T)_{L^2(T)}.
    \end{align*}
    The combination of the previous inequalities with 
    a Cauchy inequality therefore proves 
    \begin{align*}
        &\int_T f \varphi_T\,dx% \\
%         &\quad
        \lesssim 
             \ennorm{\varphi_T}_{\varepsilon,\pw,T}
             \left(\ennorm{u-u_h}_{\varepsilon,\pw,T}
            +\lVert \nabla (u-I_h u_h)\rVert_{L^2(T)}\right).
    \end{align*}
    The definition of $\varphi_T$ and a scaling argument show that 
    \begin{align*}
        \min\{h_T,h_T^2/\varepsilon\}
         \ennorm{\varphi_T}_{\varepsilon,\pw,T}
        \lesssim \lVert f\rVert_{L^2(T)}.
    \end{align*}
    The combination of the previous inequalities with 
    a Cauchy inequality therefore proves 
    \begin{align*}
        \min\{h_T,h_T^2/\varepsilon\} \lVert f\rVert_{L^2(T)}
        \lesssim \ennorm{u-u_h}_{\varepsilon,\pw,T}
            +\lVert \nabla (u-I_h u_h)\rVert_{L^2(T)}
            +\osc_\varepsilon(f,\{T\}).
    \end{align*}
    
    \textbf{Efficiency of $\eta_1$.}
    For the estimation of $\eta_1$, fix $T\in\tri$ and 
    let $F\in\mathcal{F}(T)$.
    We first split the error 
    in the tangential and the normal component, i.e., 
    \begin{align*}
        \lVert [D^2_\pw u_h\nu_F]_F\rVert_{L^2(F)}
        \leq \lVert [D^2_\pw u_h\nu_F]_F\cdot\nu_F\rVert_{L^2(F)}
          + \lVert [D^2_\pw u_h\nu_F]_F\cdot\tau_F\rVert_{L^2(F)}.
    \end{align*}
    Since $h_F^{-1}\min\{\varepsilon^{3/2},h_F^{1/2}\varepsilon\}
    \leq (\varepsilon/\kappa_T)^{1/2}$, the tangential component 
    times $\varepsilon^{3/2}\kappa_T^{1/2}$
    is bounded by $\mu_\nc(T)$ through an inverse inequality.
    To bound the normal component, we     
    follow the idea of 
    \cite{GallistlTian2024}, see also \cite{Verfuerth2013} for the 
    second-order case, and 
    employ an edge bubble function $\chi_{F,\delta}$ that takes 
    the singular perturbation into account.
    To this end, let $0<\delta\leq 1$.
    The function $\chi_{F,\delta}$ is then constructed as a bubble 
    function of a triangles with height $\delta h_F$ over the face $F$
    and has the properties 
    $\chi_{F,\delta}\in H^2_0(\omega_F)$,
    $\chi_{F,\delta}\vert_F=0$, 
%     \begin{align*}
%         \chi_{F,\delta}\vert_{E}=0
%         \quad \text{and}\quad 
%         \nabla\chi_{F,\delta}\vert_{E}=0
%         \quad \text{ for all }E\in\mathcal{F}\setminus\{F\},
%     \end{align*}
    and the scaling properties 
    \begin{equation}\label{e:scaling_bubble}
        \begin{aligned}
            h_F\delta\lVert \partial\chi_{F,\delta}/\partial \nu_F\rVert_{L^\infty(F)}
          &\approx h_F^{-1}\delta^{-1/2} \lVert \chi_{F,\delta}\rVert_{L^2(\omega_F)}
          \approx \delta^{1/2} 
             \lVert \nabla \chi_{F,\delta}\rVert_{L^2(\omega_F)}\\
          &\approx h_F\delta^{3/2} 
             \lVert D^2 \chi_{F,\delta}\rVert_{L^2(\omega_F)}
          \approx 1.
        \end{aligned}
    \end{equation}
    Details can be found in \cite{Verfuerth1998,Verfuerth2013,GallistlTian2024}.

    Define $\varphi:=[D^2_\pw u_h\nu_F]_F\cdot\nu_F 
    \nabla \chi_{F,\delta}$.
    Since $\nabla \chi_{F,\delta}$ is a quadratic bubble 
    along $F$ pointing in normal direction, it follows 
    \begin{align*}
        (h_F\delta)^{-1}
           \lVert [D^2_\pw u_h\nu_F]_F\cdot\nu_F\rVert_{L^2(F)}^2
        &\lesssim 
          \int_F ([D^2_\pw u_h\nu_F]_F\cdot\nu_F) \nu_F\cdot \varphi\,ds \\
        &= \int_F ([D^2_\pw u_h\nu_F]_F)\cdot \varphi\,ds 
        = \int_{\omega_F} D_\pw^2 u_h:D\varphi\,dx\\
        &= ([D^2_\pw u_h\nu_F]_F\cdot\nu_F) 
           \int_{\omega_F} D_\pw^2 u_h:D^2\chi_{F,\delta}\,dx.
    \end{align*}
    Since $\chi_{F,\delta}\in H^2_0(\Omega)$ 
    is a suitable test function for problem \eqref{e:sing_pert_weak_form},
    we arrive at
    \begin{align*}
        & h_F^{-1/2}\delta^{-1} \, 
           \lVert [D^2_\pw u_h\nu_F]_F\cdot\nu_F\rVert_{L^2(F)}\\
        &\qquad 
        \lesssim 
             \int_{\omega_F} D_\pw^2 (u_h-u):D^2\chi_{F,\delta}\,dx
           + \varepsilon^{-2} \int_{\omega_F} f \,\chi_{F,\delta}\,dx\\
        &\qquad \qquad \qquad \qquad \qquad \qquad \qquad \qquad \qquad 
           + \varepsilon^{-2}
             \int_{\omega_F} \nabla u\cdot \nabla \chi_{F,\delta}\,dx.
    \end{align*}
    Since $\nabla I_h u_h$ is piecewise constant and 
    $\chi_{F,\delta}\vert_{\partial T}=0$, 
    the last term equals 
    \begin{align*}
        \varepsilon^{-2}\int_{\omega_F} \nabla u\cdot \nabla \chi_{F,\delta}\,dx
        = \varepsilon^{-2}\int_{\omega_F} \nabla (u-I_h u_h) \cdot 
              \nabla \chi_{F,\delta}\,dx.
    \end{align*}
    The combination of the foregoing displayed formulae 
    leads to 
    \begin{align*}
        \lVert [D_\pw^2 u_h\nu_F]_F\cdot\nu_F\rVert_{L^2(F)}
        &
        \lesssim 
          h_F^{1/2} \delta
          \lVert D_\pw^2(u-u_h)\rVert_{L^2(\omega_F)}\, 
             \lVert D^2\chi_{F,\delta}\rVert_{L^2(\omega_F)}\\
        &\qquad
         + \varepsilon^{-2}h_F^{1/2} \delta
           \lVert f\rVert_{L^2(\omega_F)} \,
             \lVert\chi_{F,\delta}\rVert_{L^2(\omega_F)}\\
        &\qquad
         + \varepsilon^{-2} h_F^{1/2} \delta
           \lVert \nabla(u-I_h u_h)\rVert_{L^2(\omega_F)} 
            \, \lVert \nabla\chi_{F,\delta}\rVert_{L^2(\omega_F)}.
    \end{align*}
    Let $\delta:=\min\{1,\varepsilon/h_F\}$. 
    Then $\kappa_T/\delta \approx h_F/\varepsilon$.
    This and 
    the scaling of the bubble function from \eqref{e:scaling_bubble}
    then prove
    \begin{align*}
        &\varepsilon^{3/2}\kappa_T^{1/2}
        \lVert [D_\pw^2 u_h\nu_F]_F\cdot\nu_F\rVert_{L^2(F)}\\
        &\qquad 
        \lesssim \varepsilon \lVert D_\pw^2(u-u_h)\rVert_{L^2(\omega_F)}
         + \varepsilon\kappa_T^2
           \lVert f\rVert_{L^2(\omega_F)} 
         + \kappa_T 
           \lVert \nabla(u-I_h u_h)\rVert_{L^2(\omega_F)}.
    \end{align*}
    Since $\varepsilon\kappa_T \leq h_T$ and $\kappa_T\leq 1$, the efficiency 
    of $h_T\kappa_T\lVert f\rVert_{L^2(T)}$ 
    therefore proves 
    \begin{align*}
        &\varepsilon^{3/2}\kappa_T^{1/2}
        \lVert [D_\pw^2 u_h\nu_F]_F\cdot\nu_F\rVert_{L^2(F)}\\
        &\qquad 
        \lesssim \ennorm{u-u_h}_{\varepsilon,\pw,\omega_F}
         + \lVert \nabla(u-I_h u_h)\rVert_{L^2(\omega_F)}
         + \osc_\varepsilon(f,\{T_+,T_-\}).
    \end{align*}
    
    \textbf{Efficiency of $\eta_2$.}
    For the efficiency of 
    $h_T^{1/2}\kappa_T
    \,\lVert[\nabla I_h u_h\cdot\nu_F]_F\rVert_{L^2(F)}$,
    define the function $\flat_F=\lambda_a^2\lambda_b^2\in H^2_0(\omega_F)$
    and let $\varphi:=[\nabla I_h u_h\cdot\nu_F]_F\flat_F$.
    Then 
    \begin{align*}
        \lVert[\nabla I_h u_h\cdot\nu_F]_F\rVert_{L^2(F)}^2
        & \lesssim \int_F [\nabla I_h u_h\cdot\nu_F]_F \,\varphi\,ds
        = \int_{\omega_F} \nabla I_h u_h \cdot\nabla \varphi\,dx\\
        &= \int_{\omega_F} \nabla (I_h u_h-u) \cdot\nabla \varphi\,dx
          + \int_{\omega_F} \nabla u \cdot\nabla \varphi\,dx.
    \end{align*}
    Since $\varphi\in H^2_0(\omega_F)$ is a suitable test function
    in \eqref{e:sing_pert_weak_form}, the
    last term equals 
    \begin{align*}
        &\int_{\omega_F} \nabla u \cdot\nabla \varphi\,dx
         = \int_{\omega_F} f \,\varphi\,dx
           - \varepsilon^2 \int_{\omega_F} D^2 u :D^2\varphi\,dx\\
        &\qquad
        = \int_{\omega_F} f \,\varphi\,dx
           - \varepsilon^2 \int_{\omega_F} D^2_\pw (u-u_h) :D^2\varphi\,dx
           - \varepsilon^2 \int_{\omega_F} D^2_\pw u_h :D^2\varphi\,dx
           .
    \end{align*}
    Piecewise integration by parts for the last term leads to 
    \begin{align*}
        - \varepsilon^2 \int_{\omega_F} D^2_\pw u_h :D^2\varphi\,dx
        &= -\varepsilon^2 \int_F [D^2_\pw u_h \nu_F]_F \cdot \nabla \varphi\,ds\\
        &\leq \varepsilon^2 \lVert [D^2_\pw u_h \nu_F]_F\rVert_{L^2(F)}
           \lVert \nabla \varphi\rVert_{L^2(F)}.
    \end{align*}
    The scaling of $\varphi$ reads 
    \begin{align*}
        h_T^{1/2}\lVert \nabla \varphi\rVert_{L^2(F)}
        \approx
        h_T^{-1}\lVert \varphi\rVert_{L^2(\omega_F)}
        \approx
        \lVert \nabla \varphi\rVert_{L^2(\omega_F)}
        &\approx
        h_T\lVert D^2\varphi\rVert_{L^2(\omega_F)}\\
        &\approx 
        \lvert [\nabla I_h u_h\cdot\nu_F]_F\rvert .
    \end{align*}
    This together with the above displayed formulae leads to 
    \begin{align*}
        &h_T^{1/2}\lVert[\nabla I_h u_h\cdot\nu_F]_F\rVert_{L^2(F)}
        \lesssim 
        \lVert \nabla(u-I_h u_h)\rVert_{L^2(\omega_F)}
        + h_T\lVert f\rVert_{L^2(\omega_F)}\\
        &\qquad \qquad \qquad 
        + \varepsilon^2 h_T^{-1} \lVert D^2_\pw(u-u_h)\rVert_{L^2(\omega_F)}
        + \varepsilon^2 h_T^{-1/2} \lVert [D^2_\pw u_h\nu_F]_F\rVert_{L^2(F)}.
    \end{align*}
    This, $\kappa_T\leq 1$, 
    $\varepsilon h_T^{-1}\kappa_T = \min\{1,\varepsilon/h_T\}\leq 1$, 
    and 
    \begin{align*}
        \varepsilon^{1/2} h_T^{-1/2} \kappa_T
        = \min\{\varepsilon^{1/2} h_T^{-1/2}, h_T^{1/2}\varepsilon^{-1/2}\}
        \leq \kappa_T^{1/2}
    \end{align*}
    imply
    \begin{align*}
        h_T^{1/2}\kappa_T \lVert[\nabla I_h u_h\cdot\nu_F]_F\rVert_{L^2(F)}
        \lesssim \lVert \nabla(u-I_h u_h)\rVert_{L^2(\omega_F)}
           + \ennorm{u-u_h}_{\varepsilon,\pw, \omega_F}\\
           + \eta_f(T_+)+\eta_f(T_-)
           + \eta_1(T_+).
    \end{align*}
    This and the efficiency of $\eta_f$ and $\eta_1$ proves the 
    efficiency of $\eta_2$.
    
    \textbf{Efficiency of $\mu_\nc$.}
    In the case $h_T\leq \varepsilon$, the efficiency of $\mu_\nc(T)$ 
    can be proved with standard arguments, see 
    \cite{BeiraodaVeigaNiiranenStenberg2007,HuShi2009}.
    If $\varepsilon<h_T$, the multiplicative trace inequality
    shows as in \cite{GallistlTian2024} that 
    \begin{align*}
        &\varepsilon \lVert[\nabla_\pw u_h]_F\rVert_{L^2(F)}^2\\
        &\quad
        \lesssim \frac{\varepsilon}{h_F} \lVert\nabla_\pw (u-u_h)\rVert_{L^2(\omega_F)}^2
           + \varepsilon \lVert \nabla_\pw(u-u_h)\rVert_{L^2(\omega_F)}
             \,\lVert D_\pw^2(u-u_h)\rVert_{L^2(\omega_F)}\\
        &\quad 
        \lesssim \varepsilon \lVert D_\pw^2(u-u_h)\rVert_{L^2(\omega_F)}^2
           + \lVert \nabla_\pw(u-u_h)\rVert_{L^2(\omega_F)}.
    \end{align*}
\end{proof}

\section{The von K\'arm\'an equations}
\label{s:vK}

This section is devoted to the nonlinear von K\'arm\'an equations 
posed on the polygonal, bounded Lipschitz-domain $\Omega\subseteq \R^2$.
Section~\ref{ss:vK_cont} introduces the weak formulation, 
while Section~\ref{ss:vK_discrete} introduces the discrete 
problem and states an a~priori error estimate.
Section~\ref{ss:vK_errorest} contains the main part, namely the 
definition of the error estimator and the proof of its 
reliability and efficiency.

In this section, let $H^k(\Omega;\R^2)$ (resp.\ $L^p(\Omega;\R^2)$)
denote the vector valued space $(H^k(\Omega))^2$
(resp.\ $(L^p(\Omega))^2$).

\subsection{The continuous problem}\label{ss:vK_cont}

Define the bilinear form $A:H^2(\Omega;\R^2)\times H^2(\Omega;\R^2)\to\R$ 
and the trilinear form 
$B:H^2(\Omega;\R^2)\times H^1(\Omega;\R^2)\times H^1(\Omega;\R^2)\to\R$
for $\Psi=(\psi_1,\psi_2)$, $\Phi=(\varphi_1,\varphi_2)$, and 
$\Theta=(\theta_1,\theta_2)$ by 
\begin{align*}
    A(\Psi,\Phi) &= a(\psi_1,\varphi_1) + a(\psi_2,\varphi_2) ,
    \qquad \text{for all }\Psi,\Phi\in H^2(\Omega;\R^2),
\end{align*}
and 
\begin{align*}
    B(\Theta,\Psi,\Phi) &= b(\theta_1,\psi_2,\varphi_1) + b(\theta_2,\psi_1,\varphi_1)
        - b(\theta_1,\psi_1,\varphi_2)
\end{align*}
for all $\Theta\in H^2(\Omega;\R^2), \Psi,\Phi\in H^1(\Omega;\R^2)$,
where the bilinear and trilinear forms 
$a:H^2(\Omega)\times H^2(\Omega)\to\R$ and 
$b:H^2(\Omega)\times H^1(\Omega)\times H^1(\Omega)$ are defined by 
\begin{align*}
    a(\psi,\varphi) &= \int_\Omega D^2 \psi:D^2\varphi\,dx,
    \qquad \text{for all }\psi,\phi\in H^2(\Omega),\\
    b(\theta,\psi,\varphi) 
       &= \frac{1}{2}\int_\Omega
            (\cof(D^2\theta)\nabla \psi)\cdot\nabla \varphi\,dx
    \qquad \text{for all }\theta\in H^2(\Omega), \psi,\varphi\in H^1(\Omega).
\end{align*}
Furthermore, define the right-hand side $F:L^2(\Omega;\R^2)\to\R$ by 
\begin{align*}
    F(\Phi):= \int_\Omega f\,\varphi_1\,dx
    \qquad \text{for all }\Phi=(\varphi_1,\varphi_2)\in L^2(\Omega;\R^2).
\end{align*}
The continuous problem seeks $\Psi\in H^2_0(\Omega;\R^2)$ with 
\begin{align}\label{e:vK_cont_problem}
    A(\Psi,\Phi) + B(\Psi,\Psi,\Phi) = F(\Phi)
    \qquad \text{for all }\Phi\in H^2_0(\Omega;\R^2).
\end{align}
This problem is equivalent to $N(\Psi,\bullet)=0$ on $H^2_0(\Omega;\R^2)$
with  
\begin{align*}
    N(\Psi;\Phi):=A(\Psi,\Phi)+B(\Psi,\Psi,\Phi)-F(\Phi)
\end{align*}
with derivatives 
\begin{subequations}
\begin{align}
    DN(\Psi;\Xi,\Phi) &= \langle A\Xi + B'(\Psi)\Xi,\Phi\rangle,\\
    \langle B'(\Psi)\Xi,\Phi\rangle&= 2 B(\Psi,\Xi,\Phi),\\
    \label{e:D2N}
    D^2N(\Psi;\Xi,\Xi,\Phi) &= \langle B''(\Xi,\Xi),\Phi\rangle 
       = 2B(\Xi,\Xi,\Phi).
\end{align}
\end{subequations}

The following theorem states the existence and regularity 
of an exact solution.

\begin{theorem}
    There exists at least one solution $\Psi\in H^2_0(\Omega;\R^2)$ 
    to \eqref{e:vK_cont_problem}.
    If $f$ is sufficiently small, the solution is unique.
    Moreover, there exists 
    a positive parameter $\alpha$
    (with $1/2<\alpha\leq 1$ in the case of clamped boundary conditions)
    such that $\Psi\in H^{2+\alpha}(\Omega;\R^2)$ and 
    \begin{align}\label{e:vK_stab}
        \lVert \Psi\rVert_{H^{2+\alpha}(\Omega)} 
        \lesssim \lVert f\rVert_{L^2(\Omega)}.
    \end{align}
\end{theorem}

\begin{proof}
    The proof of the existence is contained in 
    \cite[Theorem~5.8-3, pp.~416]{Ciarlet_bookII_2022} 
    and \cite[Theorem~4, pp.~238]{Knightly1967}.
    The proof of the uniqueness is contained in 
    \cite[Theorem~5]{Knightly1967}.
    The regularity theory for the biharmonic problem 
    \cite{BlumRannacher1980,Grisvard1992} proves 
    \eqref{e:vK_stab}.
\end{proof}

Throughout the remaining parts of the paper, we assume, 
that the exact solution $\Psi\in H^2_0(\Omega;\R^2)$ is 
\emph{nonsingular}, that is \cite[Def.~2.4, pp.466]{Keller1975}
the Fr\'echet derivative is nonsingular. This then implies
the continuous inf-sup condition of the linearized form 
\begin{align}\label{e:vK_inf-sup}
    \lVert D^2\Phi\rVert_{L^2(\Omega)}
    \lesssim \sup_{\Theta\in H^2_0(\Omega;\R^2)\setminus\{0\}} 
         \frac{DN(\Psi;\Phi,\Theta)}{\lVert D^2\Theta\rVert_{L^2(\Omega)}}.
\end{align}

The remaining part of this section summarizes some 
frequently used inequalities.
For any $1\leq q <\infty$ the following embedding holds for 
$\Omega\subseteq \R^2$
\begin{align}\label{e:embedding}
    H^1_0(\Omega)\hookrightarrow L^q(\Omega),
    \qquad 
    \lVert v\rVert_{L^q(\Omega)}
    \leq C(q)\lVert \nabla v\rVert_{L^2(\Omega)}.
\end{align}
Furthermore, Sobolev's inequality \cite{AdamsFournier2003} states that
\begin{align}\label{e:sobolev}
    \| D^2 v\|_{L^q(\Omega)}
    \leq C(\beta, q)
    \| v\|_{H^{2+\beta}(\Omega)}
    \quad \text{for }
    0\leq\beta<1 \text{ and } 2\leq q\leq\frac{2}{1-\beta}
\end{align}
and 
\begin{align}\label{e:sobolev_infty}
    \lVert v\rVert_{L^\infty(\Omega)}
    \leq C(\beta) \lVert v\rVert_{H^{1+\beta}(\Omega)}
    \quad \text{for }
    \beta>0.
\end{align}
The regularity index $\alpha\geq 1/2$ therefore implies
\begin{align}\label{e:embedding_L4}
    \| D^2 v\|_{L^4(\Omega)}
    \lesssim
    \| v\|_{H^{2+\alpha}(\Omega)}
\end{align}
on Lipschitz polygons.

The $\tri$-piecewise version of $B$ where the Hessian and the gradient 
are replaced by their piecewise conterparts is denoted by $B_\pw$
(see also Section~\ref{ss:vK_discrete}).
\begin{lemma}[upper bounds of $B$ and its variants] \label{l:vK_cont_B}
    The following inequalities hold 
    \begin{align}
        B(\Theta,\Psi,\Phi)
            &\lesssim \lVert\Theta\rVert_{H^{2+\alpha}(\Omega)}
                     \,\lVert\nabla\Psi\rVert_{L^2(\Omega)}
                     \,\lVert\nabla\Phi\rVert_{L^q(\Omega)}
        \label{e:vK_cont_B_2alpha2q}
    \end{align}
    for all $\Theta\in H^{2+\alpha}(\Omega;\R^2), 
               \Psi\in H^1(\Omega;\R^2),\Phi\in W^{1,q}(\Omega;\R^2)$
    and $\alpha/2\leq q<\infty$.
    Moreover,
    \begin{align}
        B_\pw(\Theta,\Psi,\Phi)
            & \lesssim \lVert D^2_\pw\Theta\rVert_{L^2(\Omega)}
                     \,\lVert\nabla\Psi\rVert_{L^p(\Omega)}
                     \,\lVert D^2\Phi\rVert_{L^2(\Omega)}
        \label{e:vK_cont_B_2p2}\\
            &\lesssim \lVert D^2_\pw\Theta\rVert_{L^2(\Omega)}
                     \,\lVert D^2\Psi\rVert_{L^2(\Omega)}
                     \,\lVert D^2\Phi\rVert_{L^2(\Omega)}
        \label{e:vK_cont_B_222}
    \end{align}
    for all $\Theta\in H^2(\tri;\R^2),\Psi,\Phi\in H^2(\Omega;\R^2)$,
    $p>2$ (where the constant depends on $p$). Furthermore,
    \begin{align}
        B_\pw(\Theta,\Psi,\Phi)
            & \lesssim \lVert D^2_\pw\Theta\rVert_{L^2(\Omega)}
                     \,\lVert D^2\Psi\rVert_{L^2(\Omega)}
                     \,\lVert \nabla\Phi\rVert_{L^q(\Omega)}
        \label{e:vK_cont_B_22q}
    \end{align}
    for all $\Theta\in H^{2}(\tri;\R^2), \Psi,\Phi\in H^2(\Omega;\R^2)$, 
    $q>2$ (where the constant depends on $q$), and
    \begin{align}
        B_\pw(\Theta,\Psi,\Phi)
            & \lesssim \lVert D^2_\pw\Theta\rVert_{L^2(\Omega)}
                     \,\lVert \nabla \Psi\rVert_{L^2(\Omega)}
                     \,\lVert \nabla\Phi\rVert_{L^\infty(\Omega)}
        \label{e:vK_cont_B_22infty}
    \end{align}
    for all $\Theta\in H^{2}(\tri;\R^2), \Psi\in H^1(\Omega;\R^2), 
    \Phi\in W^{1,\infty}(\Omega;\R^2)$.
\end{lemma}

\begin{proof}
    The proof follows from H\"older inequalities and the 
    embedding \eqref{e:embedding} and the Sobolev inequality \eqref{e:sobolev}.
\end{proof}

\subsection{The discretization and a~priori error analysis} \label{ss:vK_discrete}

The discrete problem employs the discrete multilinear forms 
$A_\pw:H^2(\tri;\R^2)\times H^2(\tri;\R^2)\to\R$ and
$B_h:H^2(\tri;\R^2)\times H^1(\tri;\R^2)\times H^1(\tri;\R^2)\to\R$
with 
\begin{align*}
    a_\pw(\psi,\varphi) &= \int_\Omega D_\pw^2 \psi:D_\pw^2\varphi\,dx,\\
    A_\pw(\Psi,\Phi) &= a_\pw(\psi_1,\varphi_1) + a_\pw(\psi_2,\varphi_2) ,\\
    b_\pw(\theta,\psi,\varphi) 
       &= \frac{1}{2}\int_\Omega
            (\cof(D_\pw^2\theta)\nabla_\pw \psi)\cdot\nabla_\pw \varphi\,dx,\\
    B_\pw(\Theta,\Psi,\Phi) &= b_\pw(\theta_1,\psi_2,\varphi_1) 
        + b_\pw(\theta_2,\psi_1,\varphi_1)
        - b_\pw(\theta_1,\psi_1,\varphi_2),\\    
    B_h(\Theta,\Psi,\Phi) &= b_\pw(\theta_1,I_h\psi_2,I_h\varphi_1) 
        + b_\pw(\theta_2,I_h\psi_1,I_h\varphi_1)
        - b_\pw(\theta_1,I_h\psi_1,I_h\varphi_2),
\end{align*}
where $I_h$ denotes the nodal interpolation operator, see also 
Section~\ref{s:prelim}.
If we let $I_h$ act component-wise, we see the relation 
$B_h(\Theta,\Psi,\Phi)=B_\pw(\Theta,I_h\Psi,I_h\Phi)$.
The discrete problem seeks $\Psi_h\in \mathcal{M}_0(\tri;\R^2)$ with 
\begin{align}\label{e:karmandiscrete}
    A_\pw(\Psi_h,\Phi_h) + B_h(\Psi_h,\Psi_h,\Phi_h) = F(\Phi_h)
    \qquad \text{for all }\Phi_h\in \mathcal{M}_0(\tri;\R^2).
\end{align}

\begin{remark}
Alternatively, the right-hand side in the discrete system
\eqref{e:karmandiscrete} can be modified to $F(I_h\Phi_h)$
as in the prior section. Due to the approximation properties
of $I_h$, the arguments of the subsequent error analysis
apply to the modified formulation, too.
\end{remark}

The following theorem proves the existence of a discrete
solution of \eqref{e:karmandiscrete} and an a~priori error estimate. 
The error estimate will be employed in the 
a~posteriori error analysis below.

\begin{theorem}\label{t:vK_apriori}
    For sufficiently small $h$, there exists a (locally unique) discrete 
    solution $\Psi_h\in \mathcal{M}_0(\tri;\R^2)$ to 
    \eqref{e:karmandiscrete} and it satisfies 
    \begin{align*}
        \lVert D^2_\pw (\Psi-\Psi_h )\rVert_{L^2(\Omega)}
        \lesssim h^\alpha,
    \end{align*}
    with $\alpha$ from \eqref{e:vK_stab}.
\end{theorem}

\begin{proof}
    The proof follows similar as in \cite{MallikNataraj2016} 
    and is outlined in the appendix.
\end{proof}

The a~posteriori error estimates in Theorems~\ref{t:vK_reliability}
and \ref{t:vK_efficiency} below will contain the term 
\begin{align*}
    \lVert\nabla(\Psi-I_h\Psi_h)\rVert_{L^2(\Omega)}
\end{align*}
as part of the total error.
The following theorem shows that the predicted convergence rate
for this term is (at least) as good as that predicted for the $H^2$ error 
$\lVert D^2_\pw (\Psi-\Psi_h)\rVert_{L^2(\Omega)}$.

\begin{remark}
    The following inverse inequality is used, e.g., in the
    proof of Theorem~\ref{t:a_priori_nodal_interpol} below.
    Let the function $v_h$ be piecewise polynomial with respect to the mesh
    $\mathcal T$. A standard scaling argument shows the
    discrete inequality
        \begin{align}\label{e:inverse_inequality}
            \lVert v_h\rVert_{L^q(\Omega)}
            \lesssim \left\lVert h_\tri^{(2-q)/q} v_h\right\rVert_{L^2(\Omega)}
            \quad\text{for any }
            2\leq q <\infty,
        \end{align}
    where the constant hidden in the notation $\lesssim$ depends on the
    polynomial degree.
\end{remark}

\begin{theorem}[$H^1$ error estimate for nodal interpolant]
    \label{t:a_priori_nodal_interpol}
    For sufficiently small $h$, the discrete solution $\Psi_h$ 
    satisfies the error estimate 
    \begin{align*}
        \lVert\nabla(\Psi-I_h\Psi_h)\rVert_{L^2(\Omega)}
        \lesssim h.
    \end{align*}
\end{theorem}

\begin{proof}
    The proof follows similar as in \cite[Theorem~4.7]{MallikNataraj2016}
    and only the differences are given here.
    
    The triangle inequality 
    leads for $\rho_h = \Psi_h-I_\mathcal{M}\Psi$ to 
    \begin{align*}
        \lVert \nabla(\Psi-I_h\Psi_h)\rVert_{L^2(\Omega)}
        &\leq \lVert\nabla(\Psi-I_hI_\mathcal{M}\Psi)\rVert_{L^2(\Omega)}
          + \lVert \nabla (I_h\rho_h- \rho_h)\rVert_{L^2(\Omega)}\\
        &\qquad\qquad \qquad\quad
          + \lVert \nabla (\rho_h-J\rho_h)\rVert_{L^2(\Omega)}
          + \lVert \nabla J\rho_h\rVert_{L^2(\Omega)}.
    \end{align*}
    The approximation and stability properties
    of the nodal interpolation operator from \eqref{e:nodal_interpolation_approx},
    the Morley interpolation operator from \eqref{e:estimate_Morley_interpolation1},
    and the enriching operator from \eqref{e:enriching_stab_approx}
    together with the stability of the continuous and discrete 
    problem bound the first three terms on the right-hand side 
    by $h$.
    The last term can be bounded as in \cite[Theorem~4.7]{MallikNataraj2016}
    by $h$ plus the additional term
    \begin{align*}
        B_\pw(\Psi_h,I_h\Psi_h,I_h\zeta) 
        - B_\pw(\Psi_h,\Psi_h,I_\mathcal{M}\zeta)
    \end{align*}
    where $\zeta\in H^{2+\alpha}(\Omega)$ is the solution of a dual problem.
    Since $I_h\zeta=I_h I_\mathcal{M}\zeta$, this term can be estimated via
    \begin{align*}
        &B_\pw(\Psi_h,I_h\Psi_h,I_h\zeta) 
        - B_\pw(\Psi_h,\Psi_h,I_\mathcal{M}\zeta)\\
        &\qquad 
        = B_\pw(\Psi_h,I_h\Psi_h-\Psi_h,I_h\zeta-\zeta) 
          + B_\pw(\Psi_h,I_h\Psi_h-\Psi_h,\zeta)\\
        &\qquad\qquad\qquad 
          + B_\pw(\Psi_h,\Psi_h-\Psi,I_hI_\mathcal{M}\zeta - I_\mathcal{M}\zeta)
          + B_\pw(\Psi_h,\Psi,I_hI_\mathcal{M}\zeta - I_\mathcal{M}\zeta)\\
        &\qquad 
        \lesssim \lVert\nabla(\Psi_h-I_h\Psi_h)\rVert_{L^4(\Omega)}
             \lVert \nabla (\zeta-I_h\zeta)\rVert_{L^4(\Omega)}\\
        &\qquad\qquad\qquad 
            + \lVert \nabla(\Psi_h - I_h\Psi_h)\rVert_{L^2(\Omega)} 
              \lVert \nabla \zeta\rVert_{L^\infty(\Omega)} \\
        &\qquad\qquad\qquad 
            + \lVert \nabla(\Psi - \Psi_h)\rVert_{L^4(\Omega)} 
              \lVert \nabla 
                 (I_\mathcal{M}\zeta-I_h I_\mathcal{M}\zeta)\rVert_{L^4(\Omega)} \\
        &\qquad\qquad\qquad 
            +\lVert \nabla \Psi\rVert_{L^\infty(\Omega)} 
              \lVert \nabla (I_\mathcal{M}\zeta-I_h I_\mathcal{M}\zeta)\rVert_{L^2(\Omega)},
    \end{align*}
    where the stability of the continuous and discrete system
    was employed in the last step.
    The inverse inequality \eqref{e:inverse_inequality},
    the approximation properties of the nodal interpolation 
    operator, the embeddings \eqref{e:embedding_L4} 
    and \eqref{e:embedding}, and the a~priori error
    estimate from Theorem~\ref{t:vK_apriori}
    yield
    \begin{align*}
        \lVert\nabla(\Psi_h-I_h\Psi_h)\rVert_{L^4(\Omega)}
        +\lVert \nabla 
                 (I_\mathcal{M}\zeta-I_h I_\mathcal{M}\zeta)\rVert_{L^4(\Omega)}
        &\lesssim h^{1/2},\\
        \lVert \nabla (\zeta-I_h\zeta)\rVert_{L^4(\Omega)}
        &\lesssim h \lVert D^2\zeta\rVert_{L^4(\Omega)}
        \lesssim h\lVert\zeta\rVert_{H^{2+\alpha}(\Omega)},\\
        \lVert \nabla_\pw(\Psi - \Psi_h)\rVert_{L^4(\Omega)} 
        &\lesssim \lVert D^2_\pw(\Psi - \Psi_h)\rVert_{L^2(\Omega)}
         \lesssim h^\alpha.
    \end{align*}
    Since $\alpha\geq 1/2$,
    this together with the approximation properties 
    \eqref{e:nodal_interpolation_approx}
    and \eqref{e:sobolev_infty} 
    eventually bounds the additional terms by $h$.
\end{proof}

\subsection{Error estimator} \label{ss:vK_errorest}

Define the edge contributions
\begin{align}\label{e:A1A2}
 \begin{aligned}
A_1 &:= \cof(D_\pw^2 \psi_{h,1})\nabla I_h\psi_{h,2}
              + \cof(D_\pw^2 \psi_{h,2})\nabla I_h\psi_{h,1}
\\
\text{and }
A_2&:=\cof(D_\pw^2 \psi_{h,1})\nabla I_h\psi_{h,1}
.
\end{aligned}
\end{align}
We define the error estimator contributions
\begin{align*}
    \mu_\nc(T) &:= \sqrt{\sum_{F\in\mathcal{F}(T)} h_T^{-1}
            \lVert [\nabla_\pw \Psi_h]_F\rVert_{L^2(F)}^2},\\
    \mu_{I_h}(T) &:= \lVert \nabla_\pw (\Psi_h-I_h\Psi_h)\rVert_{L^2(T)},\\
    \eta_f(T) &:= \lVert h_T^2 f\rVert_{L^2(T)},
    \\
    \eta_j(T) &:= \sqrt{\sum_{F\in\mathcal{F}(T)\cap\mathcal{F}(\Omega)}
            h_T^{3} \lVert [A_j]_F\nu_F\rVert_{L^2(F)}^2}
            \quad\text{for }j=1,2 ,
\end{align*}
and the global error estimator by 
\begin{align*}
    \eta:=\left(\sum_{T\in\tri}\left(\mu_\nc^2(T)+\mu_{I_h}^2(T) 
      +\eta_f^2(T) + \eta_1^2(T) + \eta_2^2(T)\right)\right)^{1/2}.
\end{align*}

The following theorem proves the reliability of the error 
estimator, while the efficiency is proved in 
Theorem~\ref{t:vK_efficiency} below.

\begin{theorem}[reliability]\label{t:vK_reliability}
    If $h_\mathrm{max}$ is sufficiently small, then the error estimator is 
    reliable in the sense that
    \begin{align*}
        \lVert D_\pw^2 (\Psi-\Psi_h)\rVert_{L^2(\Omega)}
        + \lVert \nabla (\Psi-I_h\Psi_h)\rVert_{L^2(\Omega)}
        \lesssim \eta.
    \end{align*}
\end{theorem}

\begin{proof}
    Let $J$ denote the enriching operator from Section~\ref{s:prelim}.
    The inf-sup condition \eqref{e:vK_inf-sup} of the linearized 
    problem guarantees the existence 
    of $\Theta\in H^2_0(\Omega;\R^2)$ with 
    $\lVert D^2 \Theta \rVert_{L^2(\Omega)}=1$ and
    \begin{align*}
        \lVert D^2(\Psi-J\Psi_h)\rVert_{L^2(\Omega)}
        \lesssim DN(\Psi;\Psi-J\Psi_h,\Theta).
    \end{align*}
    Since $N$ is quadratic and $\Psi$ is the exact solution
    with $N(\Psi,\cdot)=0$, 
    the following Taylor expansion around $\Psi$ 
    is exact,
    see also \cite{CarstensenMallikNataraj2020},
    \begin{align*}
        N(J\Psi_h;\Theta)
        = DN(\Psi;J\Psi_h-\Psi;\Theta) 
          + \frac{1}{2} D^2N(\Psi;J\Psi_h-\Psi,J\Psi_h-\Psi,\Theta).
    \end{align*}
    Combining the two foregoing formulas with the representation
    \eqref{e:D2N} results in
    \begin{align*}
        \lVert D^2(\Psi-J\Psi_h)\rVert_{L^2(\Omega)}
        \lesssim -N(J\Psi_h;\Theta) + B(J\Psi_h-\Psi,J\Psi_h-\Psi,\Theta).
    \end{align*}
    We apply estimate \eqref{e:vK_cont_B_222} to the right-hand side
    of this estimate, which leads to
    \begin{align*}
        \lVert D^2(\Psi-J\Psi_h)\rVert_{L^2(\Omega)}
        \lesssim -N(J\Psi_h;\Theta) + \lVert D^2( J\Psi_h-\Psi)\rVert_{L^2(\Omega)}^2.
    \end{align*}
    From the a~priori error analysis it is known that 
    $J\Psi_h-\Psi\to 0$ for decreasing mesh size, and since the error 
    term appears quadratically on the right-hand side, eventually we
    deduce under the assumption $h_{\max}\ll1$ that
    \begin{align*}
        \lVert D^2(\Psi-J\Psi_h)\rVert_{L^2(\Omega)}
        \lesssim -N(J\Psi_h;\Theta) .
    \end{align*}
    
    The definition of $N$ and the 
    discrete equation \eqref{e:karmandiscrete} imply for the Morley interpolation
    operator $I_\mathcal{M}$ from Section~\ref{s:prelim} that
    \begin{align*}
        -N(J\Psi_h;\Theta) 
        &= -A(J\Psi_h,\Theta) - B(J\Psi_h,J\Psi_h,\Theta) +F(\Theta)\\
        &=  A_\pw(\Psi_h,I_\mathcal{M}\Theta) - A(J\Psi_h,\Theta)
            +F(\Theta-I_\mathcal{M}\Theta)\\
        &\qquad 
        + B_h(\Psi_h,\Psi_h,I_\mathcal{M}\Theta)- B(J\Psi_h,J\Psi_h,\Theta)
           .
    \end{align*}
    The integral mean property of the Morley interpolation 
    operator \eqref{e:integral_mean_property}, the normalization 
    $\lVert D^2\Theta\rVert_{L^2(\Omega)}=1$, and \eqref{e:enriching_apost} imply
    \begin{align*}
        A_\pw(\Psi_h,I_\mathcal{M}\Theta)-A(J\Psi_h,\Theta)
        &=  A_\pw(\Psi_h,\Theta) - A(J\Psi_h,\Theta)
        = A_\pw(\Psi_h-J\Psi_h,\Theta) \\
        &\leq \lVert D^2_\pw(\Psi_h-J\Psi_h)\rVert_{L^2(\Omega)}
        \lesssim \eta.
    \end{align*}
    Furthermore, the approximation properties of 
    $I_\mathcal{M}$ from \eqref{e:estimate_Morley_interpolation1}
    and $\lVert D^2\Theta\rVert_{L^2(\Omega)}=1$ imply 
    \begin{align*}
        F(\Theta -I_\mathcal{M}\Theta)
        = \int_\Omega f (\theta_1-I_\mathcal{M}\theta_1)\,dx
        \lesssim \lVert h_\tri^2 f\rVert_{L^2(\Omega)}
        \lesssim \eta.
    \end{align*}
    Combining the foregoing estimates results in
    \begin{align}\label{e:Nerrors_1}
          \lVert D^2(\Psi-J\Psi_h)\rVert_{L^2(\Omega)}
      \lesssim
       \eta 
        +
        \left |B_h(\Psi_h,\Psi_h,I_\mathcal{M}\Theta)- B(J\Psi_h,J\Psi_h,\Theta)
        \right|
           .
    \end{align}
    It therefore remains to estimate the terms on the right-hand side 
    of \eqref{e:Nerrors_1} involving $B$ and $B_h$.
    The definition of $B_h$ and 
    $I_h I_\mathcal{M}\Theta=I_h \Theta$ (see Section~\ref{s:prelim}) imply 
    \begin{equation}\label{e:vK_proof_reliability_1}
    \begin{aligned}
        &B_h(\Psi_h,\Psi_h,I_\mathcal{M}\Theta)
         - B(J\Psi_h,J\Psi_h,\Theta)
           \\
        &\qquad  
        =  B_\pw(\Psi_h,I_h\Psi_h,I_h\Theta)
          - B(J\Psi_h,J\Psi_h,\Theta)\\
        &\qquad  
        = B_\pw(\Psi_h-J\Psi_h,J\Psi_h,\Theta) 
          + B_\pw(\Psi_h-\Psi,I_h\Psi_h-J\Psi_h,\Theta)\\
        &\qquad \qquad 
          + B(\Psi,I_h\Psi_h-J\Psi_h,\Theta)
          - B_\pw(\Psi_h,I_h\Psi_h,\Theta-I_h\Theta).
    \end{aligned}
    \end{equation}
    Since $\lVert D^2\Theta\rVert_{L^2(\Omega)}=1$, 
    the inequality \eqref{e:vK_cont_B_222} implies for 
    the first term on the right-hand side
    of \eqref{e:vK_proof_reliability_1} that 
    \begin{align*}
        B_\pw(\Psi_h-J\Psi_h,J\Psi_h,\Theta)
        \lesssim \lVert D^2_\pw(\Psi_h-J\Psi_h)\rVert_{L^2(\Omega)}
              \,\lVert D^2 J\Psi_h\rVert_{L^2(\Omega)}.
    \end{align*}
    The stability properties of the enrichment operator and the 
    stability of the discrete problem lead to
    $\lVert D^2 J\Psi_h\rVert_{L^2(\Omega)}\lesssim 1$, and, therefore,
    \eqref{e:enriching_apost} implies
    \begin{align*}
        B_\pw(\Psi_h-J\Psi_h,J\Psi_h,\Theta) 
        \lesssim \sqrt{\sum_{T\in\tri}\mu_\nc(T)^2}
        \leq\eta
        .
    \end{align*}

    For the second term of the right-hand side of 
    \eqref{e:vK_proof_reliability_1} we use \eqref{e:vK_cont_B_2p2}
    to conclude that
    \begin{align*}
        &B_\pw(\Psi_h-\Psi,I_h\Psi_h-J\Psi_h,\Theta)\\
        &\qquad \qquad\qquad
        \leq C(p) \lVert D_\pw^2 (\Psi-\Psi_h)\rVert_{L^2(\Omega)}
           \lVert \nabla (J\Psi_h-I_h\Psi_h)\rVert_{L^p(\Omega)}
    \end{align*}
    for some $2<p\leq 4$.
    The inverse inequality \eqref{e:inverse_inequality} 
    and the approximation properties of $I_h$ from 
    \eqref{e:nodal_interpolation_approx} together with
    \eqref{e:prop_enriching_nodal} and the stability of the 
    enriching operator \eqref{e:enriching_stab_approx} imply
    for $\gamma:={1+(2-p)/p}\geq 1/2$ that
    \begin{align*}
        &\lVert \nabla (J\Psi_h-I_h\Psi_h)\rVert_{L^p(\Omega)}
        \lesssim
            \lVert h_\tri^{(2-p)/p}\nabla (J\Psi_h-I_hJ\Psi_h)\rVert_{L^2(\Omega)}\\
        &\qquad\qquad
        \lesssim
            \lVert h_\tri^{1+(2-p)/p}D^2 J\Psi_h \rVert_{L^2(\Omega)}
        \lesssim \lVert h_\tri^\gamma D^2_\pw \Psi_h \rVert_{L^2(\Omega)}
        \lesssim h_{\max}^\gamma
        .
    \end{align*}

    The stability of the continuous problem \eqref{e:vK_stab} 
    and \eqref{e:vK_cont_B_2alpha2q} together with 
    the embedding \eqref{e:embedding}
    lead for the third term on the right-hand side 
    of \eqref{e:vK_proof_reliability_1} to 
    \begin{align*}
        B(\Psi,I_h\Psi_h-J\Psi_h,\Theta)
        \lesssim \lVert \Psi\rVert_{H^{2+\alpha}(\Omega)} 
              \,\rVert \nabla (J\Psi_h-I_h\Psi_h)\rVert_{L^2(\Omega)}.
    \end{align*}
    A triangle and a discrete Poincar\'e inequality 
    \cite{Brenner2003} (note that the integral mean of 
    $[\nabla\Psi_h]_F$ vanishes on each face) then imply 
    \begin{align*}
        \lVert \nabla (J\Psi_h-I_h\Psi_h)\rVert_{L^2(\Omega)}
        &\leq \lVert\nabla_\pw (\Psi_h-J\Psi_h)\rVert_{L^2(\Omega)}
            + \lVert\nabla_\pw (\Psi_h-I_h\Psi_h)\rVert_{L^2(\Omega)}\\
        &\leq \lVert D_\pw^2 (\Psi_h-J\Psi_h)\rVert_{L^2(\Omega)}
            + \eta
        \lesssim \eta.
    \end{align*}
    
    Since $\cof(D_\pw^2\psi_{h,j})\nabla I_h\psi_{h,k}$ is piecewise 
    constant for $j,k\in\{1,2\}$,
    a piecewise integration by parts leads for the fourth term 
    in \eqref{e:vK_proof_reliability_1}
    to 
    \begin{align*}
        &-B_\pw(\Psi_h,I_h\Psi_h,\Theta-I_h\Theta)\\
        &\;
        = \frac12 \sum_{F\in\mathcal{F}(\Omega)} 
           \bigg(
               \int_F (\theta_2-I_h\theta_2)
              \,[\cof(D_\pw^2\psi_{h,1})\nabla I_h\psi_{h,1}]_F\cdot \nu_F\,ds\\
        &\quad
              -\int_F (\theta_1-I_h\theta_1)
              \,[\cof(D_\pw^2\psi_{h,1})\nabla I_h\psi_{h,2}
                 + \cof(D_\pw^2\psi_{h,2})\nabla I_h\psi_{h,1}]_F\cdot \nu_F\,ds
              \bigg).
    \end{align*}
    A Cauchy and a trace inequality and the approximation properties 
    of the nodal interpolation operator \eqref{e:nodal_interpolation_approx}
    eventually bound the right-hand side by $\eta_1+\eta_2$.

    The combination of the previous inequalities for the terms on the 
    right-hand side of \eqref{e:vK_proof_reliability_1} with 
    \eqref{e:Nerrors_1} leads to 
    \begin{align*}
        &\lVert D^2(\Psi-J\Psi_h)\rVert_{L^2(\Omega)}
        \lesssim \eta 
            + h_{\max}^\gamma\lVert D_\pw^2(\Psi-\Psi_h)\rVert_{L^2(\Omega)}
              .
    \end{align*}
    The triangle inequality and \eqref{e:enriching_apost} 
    therefore lead to 
    \begin{align*}
        \lVert D^2_\pw(\Psi-\Psi_h)\rVert_{L^2(\Omega)}
        \lesssim \eta 
            + h_{\max}^\gamma\lVert D_\pw^2(\Psi-\Psi_h)\rVert_{L^2(\Omega)}.
    \end{align*}
    For sufficiently small $h$, the last term
    on the right-hand side can be absorbed.
\end{proof}

Define the oscillations of $f$ by
\begin{align*}
    \osc(f,\tri(\omega_T))
    :=\lVert h_\tri^2(f-\Pi_0 f)\rVert_{L^2(\omega_T)}.
\end{align*}

\begin{theorem}[efficiency]\label{t:vK_efficiency}
    The error estimator is efficient in the sense that 
    \begin{align*}
        \mu_\nc(T) &\lesssim \min_{\Phi\in H^2_0(\Omega)}
                \lVert D_\pw^2(\Psi_h-\Phi)\rVert_{L^2(\Omega_T)}
                \leq \lVert D_\pw^2(\Psi-\Psi_h)\rVert_{L^2(\Omega_T)},\\
        \sum_{T\in\tri}\mu_{I_h}^2(T)
          &\lesssim \lVert D_\pw^2(\Psi-\Psi_h)\rVert_{L^2(\Omega)}
              + \lVert \nabla (\Psi-I_h\Psi_h)\rVert_{L^2(\Omega)},\\
        \eta_f(T)&\lesssim \lVert D_\pw^2(\Psi-\Psi_h)\rVert_{L^2(T)} 
              + \|\nabla(\Psi-I_h\Psi_h)\|_{L^2(T)}+  \osc(f,\{T\}),\\
        \eta_j(T)&\lesssim \lVert D_\pw^2(\Psi-\Psi_h)\rVert_{L^2(\omega_T)}
              + \|\nabla (\Psi-I_h\Psi_h)\|_{L^2(\omega_T)}+  \osc(f,\tri(\omega_T)),
    \end{align*}
    for $j=1,2$.
\end{theorem}

\begin{remark}
    The error estimator $\eta_{I_h}$ is only globally 
    efficient due to the global Poincar\'e inequality.
    However, we have the local version 
    \begin{align*}
        \mu_{I_h}^2(T)
          &\lesssim \lVert \nabla_\pw(\Psi-\Psi_h)\rVert_{L^2(T)}
              + \lVert \nabla (\Psi-I_h\Psi_h)\rVert_{L^2(T)}.
    \end{align*}
\end{remark}

\begin{proof}
    \textbf{Efficiency of $\mu_\nc$.}
    The efficiency of $\mu_\nc$ follows from \eqref{e:enriching_apost}.
    
    \textbf{Efficiency of $\mu_{I_h}$.}
    Since $\nabla_\pw\Psi_h$ is a Crouzeix-Raviart function,
    the piecewise Poincar\'e inequality \cite{Brenner2003}
    \begin{align*}
        \lVert \nabla_\pw(\Psi-\Psi_h)\rVert_{L^2(\Omega)}
        \lesssim \lVert D^2_\pw(\Psi-\Psi_h)\rVert_{L^2(\Omega)}
    \end{align*}
    together with a triangle inequality proves the efficiency 
    of $\mu_{I_h}$.
    
    \textbf{Efficiency of $\eta_f$.}
    Let $\flat_T\in H^2_0(T)$ denote the $H^2$ volume bubble function
    with $\|\varphi_T\|_{L^\infty(T)}\approx1$
    and set $\varphi:=\Pi_0 f \,\flat_T$.
    Then 
    \begin{align*}
        \lVert \Pi_0 f\rVert_{L^2(T)}^2
        \lesssim \int_T \Pi_0 f \varphi\,dx 
        = \int_T f \varphi\,dx + \int_T (\Pi_0 f - f) \varphi\,dx.
    \end{align*}
    The scaling $\lVert\varphi\rVert_{L^2(\Omega)}\lesssim \lVert\Pi_0 f\rVert_{L^2(T)}$
    proves for the last term 
    \begin{align*}
        \int_T (\Pi_0 f - f) \varphi\,dx \lesssim h_T^{-2}  \lVert\Pi_0 f\rVert_{L^2(T)} \osc(f,\{T\}).
    \end{align*}
    Since $\varphi\in H^2_0(\Omega)$, the continuous problem
    \eqref{e:vK_cont_problem} implies for the second term
    for $\Phi=(\varphi,0)$
    \begin{align*}
        \int_T f \varphi\,dx
        = A(\Psi,\Phi) + B(\Psi,\Psi,\Phi).
    \end{align*}
    Since $\varphi\in H^2_0(T)$, it holds that $A_\pw(\Psi_h,\Phi)=0$.
    Moreover, since $D^2_\pw \Psi_h$ and $\nabla_\pw I_h\Psi_h$ are
    piecewise constant, we also have from integration by parts that
    $B_\pw(\Psi_h,I_h\Psi_h,\Phi)=0$. We thus obtain
    \begin{align*}
        \int_T f \varphi\,dx
        = A_\pw(\Psi-\Psi_h,\Phi) 
        + B(\Psi,\Psi,\Phi)
        - B_\pw(\Psi_h,I_h\Psi_h,\Phi)
        .
    \end{align*}
    We use the multilinearity of $B_\pw$ and compute
    \begin{align*}
        B(\Psi,\Psi,\Phi)
           - B_\pw(\Psi_h,I_h\Psi_h,\Phi)
        =B_\pw(\Psi-\Psi_h,\Psi,\Phi)
           +B_\pw(\Psi_h,\Psi-I_h\Psi_h,\Phi)
        .
    \end{align*}
    We note that $\|D^2 \Phi\|_{L^2(\Omega)}\lesssim h_T^{-2} \|\Pi_0 f\|_{L^2(T)}$
    and estimate the last two terms 
    with \eqref{e:vK_cont_B_222} and \eqref{e:vK_cont_B_22infty}
    as follows
    $$
    B_\pw(\Psi-\Psi_h,\Psi,\Phi)
    \lesssim
    \|D^2_\pw (\Psi-\Psi_h)\|_{L^2(T)}
        \,\|D^2\Psi\|_{L^2(T)}
        \,h_T^{-2}\, \|\Pi_0 f\|_{L^2(T)}
    $$
    and
    $$
    B_\pw(\Psi_h,\Psi-I_h\Psi_h,\Phi)
    \leq
    \|D^2_\pw \Psi_h\|_{L^2(T)}\,
    \|\nabla (\Psi-I_h\Psi_h)\|_{L^2(T)}\,
    \|\nabla \Phi\|_{L^\infty(T)}
    .
    $$
    From scaling we have
    $\|\nabla \Phi\|_{L^\infty(T)}
    \lesssim h_T^{-1} \|\nabla \Phi\|_{L^2(T)}
    \lesssim h_T^{-2} \|\Pi_0 f\|_{L^2(T)}$.

    \textbf{Efficiency of $\eta_1$.}
    For the efficiency of $\eta_1(T)$, recall the abbreviation
    $A_1$ from \eqref{e:A1A2}
    and 
    let $\flat_F\in H^2_0(\omega_F)$ be the $H^2$ face bubble 
    function with $\|\flat_F\|_{L^\infty(\omega_F)}\approx1$
    as in the 
    proof of Theorem~\ref{t:sp_efficiency}.
    Define $\varphi:=[A_1]_F\cdot\nu_F\,\flat_F$.
    Then 
    \begin{align*}
        \lVert [A_1\cdot\nu_F]_F\rVert_{L^2(F)}^2
        \lesssim \int_F  \varphi \, [A_1\cdot \nu_F]_F\,ds 
        = \int_{\omega_F} A_1\cdot \nabla\varphi\,dx
        = B_\pw(\Psi_h,I_h\Psi_h,\Phi)
    \end{align*}
    for $\Phi=(\varphi,0)$. 
    The continuous problem \eqref{e:vK_cont_problem} implies 
    \begin{align*}
        B_\pw(\Psi_h,I_h\Psi_h,\Phi)
        = B_\pw(\Psi_h,I_h\Psi_h,\Phi) -B(\Psi,\Psi,\Phi)-A(\Psi,\Phi) + F(\Phi).
    \end{align*}
    The first two terms are estimated via 
    \eqref{e:vK_cont_B_222} and \eqref{e:vK_cont_B_22infty}
    \begin{align*}
        &B_\pw(\Psi_h,I_h\Psi_h,\Phi) -B(\Psi,\Psi,\Phi)
        = B_\pw(\Psi_h-\Psi,\Psi,\Phi) + B_\pw(\Psi_h,I_h\Psi_h-\Psi,\Phi)\\
        &\lesssim \lVert D_\pw^2 (\Psi-\Psi_h)\rVert_{L^2(\omega_F)} 
            \,\lVert D^2 \Psi\rVert_{L^2(\omega_F)} 
            \,\lVert D^2 \Phi\rVert_{L^2(\omega_F)}  \\
        &\qquad\qquad\qquad\qquad\qquad
            + \lVert D_\pw^2 \Psi_h\rVert_{L^2(\omega_F)} 
            \,\lVert \nabla (\Psi-I_h\Psi_h)\rVert_{L^2(\omega_F)}
            \,\lVert \nabla \Phi\rVert_{L^\infty(\omega_F)}.
    \end{align*}
    Moreover, 
    \begin{align*}
        F(\Phi)\lesssim \lVert f\rVert_{L^2(\omega_F)}\,\lVert\Phi\rVert_{L^2(\Omega)}.
    \end{align*}
    Furthermore, a piecewise integration by parts
    and $\Phi=(\varphi,0)$ with $\varphi\in H^2_0(\omega_F)$ leads to 
    \begin{align*}
        &A(\Psi,\Phi)
        = A_\pw(\Psi-\Psi_h,\Phi) + A_\pw(\Psi_h,\Phi)\\
        &\quad
        = A_\pw(\Psi-\Psi_h,\Phi) 
        + \int_F \nabla\varphi\cdot[D_\pw^2 \psi_{h,1}\nu_F]_F\,ds \\
        &\quad
        \lesssim \lVert D^2_\pw (\Psi-\Psi_h)\rVert_{L^2(\omega_F)}
            \, \lVert D^2 \Phi\rVert_{L^2(\omega_F)}
        + \lVert [D_\pw^2 \psi_{h,1}\nu_F]_F\rVert_{L^2(F)}
            \,\lVert \nabla\varphi\rVert_{L^2(F)}.
    \end{align*}
    The scaling of the bubble function reads 
    \begin{align*}
        \lVert D^2\Phi\rVert_{L^2(\Omega)}
        &\approx \lVert \nabla\Phi\rVert_{L^\infty(\Omega)}
        \approx h_T^{-3/2}\lVert[A_1\cdot\nu_F]_F\rVert_{L^2(F)},\\
        \lVert \Phi\rVert_{L^2(\Omega)}
        &\approx h_T^{1/2} \lVert[A_1\cdot\nu_F]_F\rVert_{L^2(F)}\\
        \lVert \nabla\varphi\rVert_{L^2(F)}
        &\approx h_T^{-1} \lVert[A_1\cdot\nu_F]_F\rVert_{L^2(F)}.
    \end{align*}
    The combination of the above inequalities
    together with the stability of the continuous and discrete 
    problem proves 
    \begin{align*}
        h_T^{3/2}\lVert [A_1\cdot\nu_F]_F\rVert_{L^2(F)}
        &\lesssim \lVert D_\pw^2(\Psi-\Psi_h)\rVert_{L^2(\omega_F)}
        + \lVert\nabla(\Psi-I_h\Psi_h)\rVert_{L^2(\omega_F)}\\
        &\qquad \qquad
        + \lVert h_\tri^2 f\rVert_{L^2(\omega_F)}
        + h_T^{1/2} \lVert [D_\pw^2 \psi_{h,1}\nu_F]_F\rVert_{L^2(F)}.
    \end{align*}
    The efficiency estimate of $\eta_1$ therefore follows, once the 
    efficiency of the term 
    $h_T^{1/2} \lVert [D_\pw^2 \psi_{h,1}\nu_F]_F\rVert_{L^2(F)}$
    is shown.

    The proof follows similar as the proof of the efficiency of 
    $\eta_2$ in Theorem~\ref{t:sp_efficiency}: 
    We start with a split in the tangential and the normal part,
    i.e.,
    \begin{align*}
        &h_T^{1/2} \lVert [D_\pw^2 \psi_{h,1}\nu_F]_F\rVert_{L^2(F)}\\
        &\qquad\qquad
        \leq h_T^{1/2} \lVert [D_\pw^2 \psi_{h,1}\nu_F]_F\cdot\nu_F\rVert_{L^2(F)}
        + h_T^{1/2} \lVert [D_\pw^2 \psi_{h,1}\nu_F]_F\cdot\tau_F\rVert_{L^2(F)}.
    \end{align*}
    The tangential part is bounded through an inverse inequality 
    by sums of $\mu_\nc(T)$ of adjacent triangles $T$.
    To bound the normal part, let $\chi_{F}\in H^2_0(\omega_F)$ denote the 
    edge bubble function from the proof of Theorem~\ref{t:sp_efficiency}
    (with $\delta=1$) with the properties 
    $\chi_F\vert_E=0$ for all $E\in\mathcal{F}$ and 
    $\nabla\chi_F\vert_E=0$ for all $E\in\mathcal{F}\setminus\{E\}$
    and the scaling~\eqref{e:scaling_bubble}. 
    Set $\varphi=[D_\pw^2\psi_{h,1}\cdot\nu_F]_F\,\nabla\chi_F$
    and $\Phi=(\varphi,0)$. 
    The arguments in the proof of Theorem~\ref{t:sp_efficiency}
    then show with $X = (\chi,0)$ that
    \begin{align*}
        h_T^{1/2}\lVert [D_\pw^2 \psi_{h,1}\nu_F]_F\rVert_{L^2(F)}
        \lesssim h_T A_\pw(\Psi_h,X).
    \end{align*}
    The continuous problem \eqref{e:vK_cont_problem} leads to
    \begin{align}\label{e:vK_proof_efficiency_1}
        A_\pw(\Psi_h,X)
        = A_\pw(\Psi_h-\Psi,X) +
        F(X)-B(\Psi,\Psi,X).
    \end{align}
    The first term on the right-hand side is bounded by 
    \begin{align*}
        A_\pw(\Psi_h-\Psi,X)
        \lesssim h_T^{-1}\lVert D_\pw^2(\Psi-\Psi_h)\rVert_{L^2(\omega_F)},
    \end{align*}
    where the scaling from \eqref{e:scaling_bubble} was used.
    The second term on the right-hand side of \eqref{e:vK_proof_efficiency_1}
    is estimated as 
    \begin{align*}
        F(X) 
        \lesssim \lVert h_\tri f\rVert_{L^2(\omega_F)}.
    \end{align*}
    Since $\cof(D^2\psi_{h,1})$ is divergence free, $\nabla I_h \psi_{h,1}$ 
    is piecewise constant and $\chi_F\vert_E=0$ for all 
    $E\in\mathcal{F}$, it follows $B(\Psi,I_h\Psi_h,X)=0$.
    Therefore, the third term on the right-hand side of 
    \eqref{e:vK_proof_efficiency_1} is bounded as follows
    \begin{align*}
        -B(\Psi,\Psi,X)
        &= B(\Psi,I_h\Psi_h-\Psi,X)\\
        &\lesssim h_T^{-1}\,\lVert D^2\Psi\rVert_{L^2(\omega_F)}
            \,\lVert \nabla(\Psi- I_h\Psi_h)\rVert_{L^2(\omega_F)},
    \end{align*}
    where we used \eqref{e:vK_cont_B_22infty} and the scaling 
    $\lVert\nabla X\rVert_{L^\infty(\Omega)}\approx h_T^{-1}$ in the last step.
    The combination of the previously displayed formulae 
    with the stability of the system yield 
    \begin{align*}
        &h_T^{1/2}\lVert [D_\pw^2 \psi_{h,1}\nu_F]_F\rVert_{L^2(F)}\\
        &\qquad\qquad
        \lesssim \lVert D_\pw^2(\Psi-\Psi_h)\rVert_{L^2(\omega_F)}
        + \lVert \nabla(\Psi- I_h\Psi_h)\rVert_{L^2(\omega)}
        + \lVert h_\tri^2 f\rVert_{L^2(\omega_F)} .
    \end{align*}
    Together with the efficiency of the term 
    $\lVert h_\tri^2 f\rVert_{L^2(\omega_F)}$, this implies the 
    efficiency of $\eta_1$.

    \textbf{Efficiency of $\eta_2$.}
    The efficiency of $\eta_2$ follows the same lines with the 
    test function $\Phi=(0,\varphi)$ with 
    $\varphi = \flat_F\,(\cof(D_\pw^2 \psi_{h,1})\nabla I_h\psi_{h,1})$.
\end{proof}

\section{Numerical experiments}
\label{s:num}
In this section, two numerical experiments are conducted for the singularly 
perturbed biharmonic equation from Section~\ref{s:sing}
and two for the von K{\'a}rm{\'a}n equation from Section~\ref{s:vK}. 
In all the experiments, the convergence behavior of the errors and the corresponding estimators is examined under uniform and adaptive mesh refinements. 
The convergence rates of the errors are computed with respect to the Morley degrees of freedom. 
In the adaptive algorithm, D{\"o}rfler marking
\cite{Doerfler1996} with parameter $\theta = 0.25$ is used.

\subsection{Singularly perturbed biharmonic equation}

\begin{example}\label{eg_xBL}
Consider the PDE \eqref{e:intro_sing_pert_strong_form}
on the square domain
$\Omega=(0,1)^2$ with the exact solution
$u(x,y)=w(x)w(y)$, where $w$ is given \cite[Section 5.3]{GallistlTian2024} by
\[
w(t) = 
    \sin(\pi t) 
    - \pi \varepsilon\,
      \left(\cosh\left(\tfrac{1}{2\varepsilon}\right)
            - \cosh\left(\tfrac{2t-1}{2\varepsilon}\right)\right)
           /{\sinh\left(\tfrac{1}{2\varepsilon}\right)}
           .
\]
\end{example}
In this test case, the behavior of the solution depends on the parameter $\varepsilon$. 
For $\varepsilon = 1$, the solution is non-oscillatory, whereas decreasing $\varepsilon$ leads to the formation of a boundary layer. 
We compute the solution using both uniform and adaptive mesh refinement for $\varepsilon = 1$ and $\varepsilon = 10^{-2}$. 
The numerical simulations start from an initial mesh consisting of 16 uniform triangles, obtained by applying red refinement to a criss-cross mesh.

For  $\varepsilon = 1$, the solution exhibits optimal
convergence rates under both uniform and adaptive mesh refinements
(see Figure~\ref{f:xBL_fig_a} for an adaptively generated mesh).
In contrast, for $\varepsilon=10^{-2}$, the solution develops a boundary layer because the limiting solution $\sin(\pi x)\sin(\pi y)$ does not satisfy the clamped boundary condition. 
Under adaptive mesh refinement, this boundary layer is
captured by the adaptive mesh illustrated in Figure~\ref{f:xBL_fig_b}.

Figure~\ref{f:xBL_fig_c} shows optimal convergence rates for
the error $\ennorm{u-u_h}_{\varepsilon,\pw}
+\lVert \nabla_\pw(u-I_h u_h)\rVert_{L^2(\Omega)}$
and the complete estimator $\eta$ under uniform and adaptive
mesh refinement when $\varepsilon = 1$.
For $\varepsilon = 10^{-2}$, the error suffers from
the presence of the boundary layer.
In contrast, adaptive mesh refinement shows the optimal
convergence rate almost from the beginning, see Figure~\ref{f:xBL_fig_d}.
Once the local mesh-size $h_T$ becomes smaller than
$\varepsilon$, which occurs at approximately $10^4$ degrees of freedom,
the value of $\kappa_T$ changes from $\kappa_T=1$ to
$\kappa_T=h_T/\varepsilon < 1$.
This transition is reflected in the estimator
plots in Figure~\ref{f:xBL_fig_e} for uniform mesh-refinement.
In particular, in $\eta_1$, the coefficient has dominance of
$\varepsilon$, which exhibits nonnegative convergence rates
only after $10^4$ degrees of freedom.
For adaptive mesh-refinement, this preasymptotic effect is reduced
to the range of approximately $3\cdot 10^3$ degrees of freedom,
see Figure~\ref{f:xBL_fig_f}.

\begin{figure}
\centering
\begin{subfigure}{.48\textwidth}
\includegraphics[width=\textwidth]{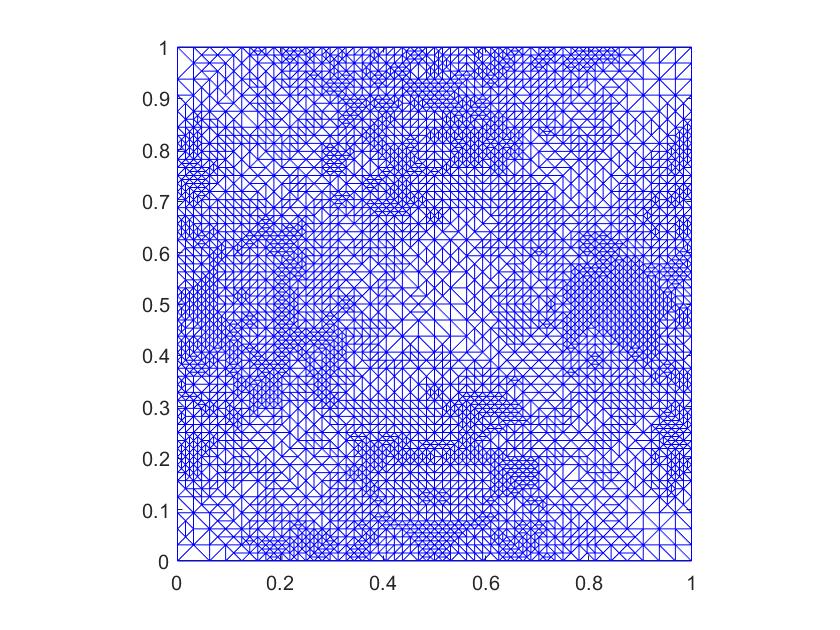} 
\caption{Adaptive mesh with $\varepsilon=1$ (NDOF 18\,291)}
\label{f:xBL_fig_a}
\end{subfigure}
\hfil
\begin{subfigure}{.48\textwidth}
\includegraphics[width=\textwidth]{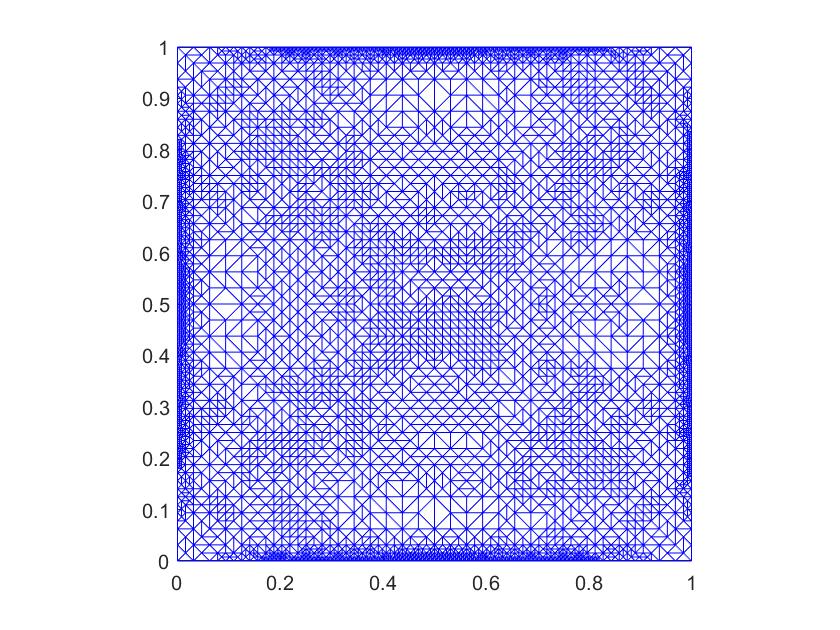}
\caption{Adaptive mesh with $\varepsilon=10^{-2}$ (NDOF 15\,677)}
\label{f:xBL_fig_b}
\end{subfigure}

\begin{subfigure}{.48\textwidth}
\includegraphics[width=\textwidth]{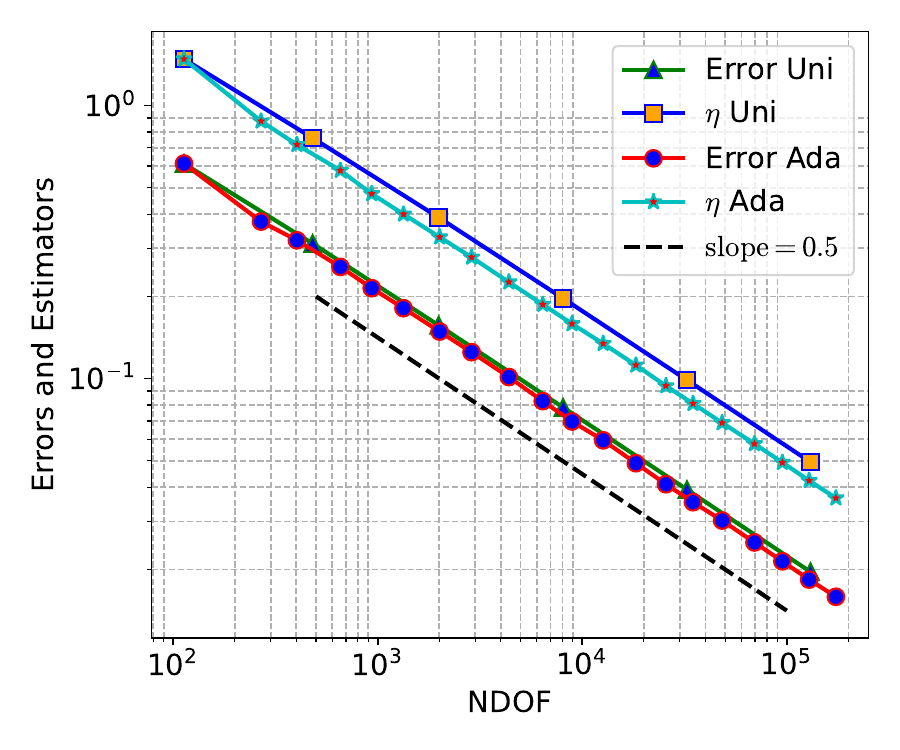}
\caption{Errors and estimators with $\varepsilon=1$}
\label{f:xBL_fig_c}
\end{subfigure} 
\hfil
\begin{subfigure}{.48\textwidth}
\includegraphics[width=\textwidth]{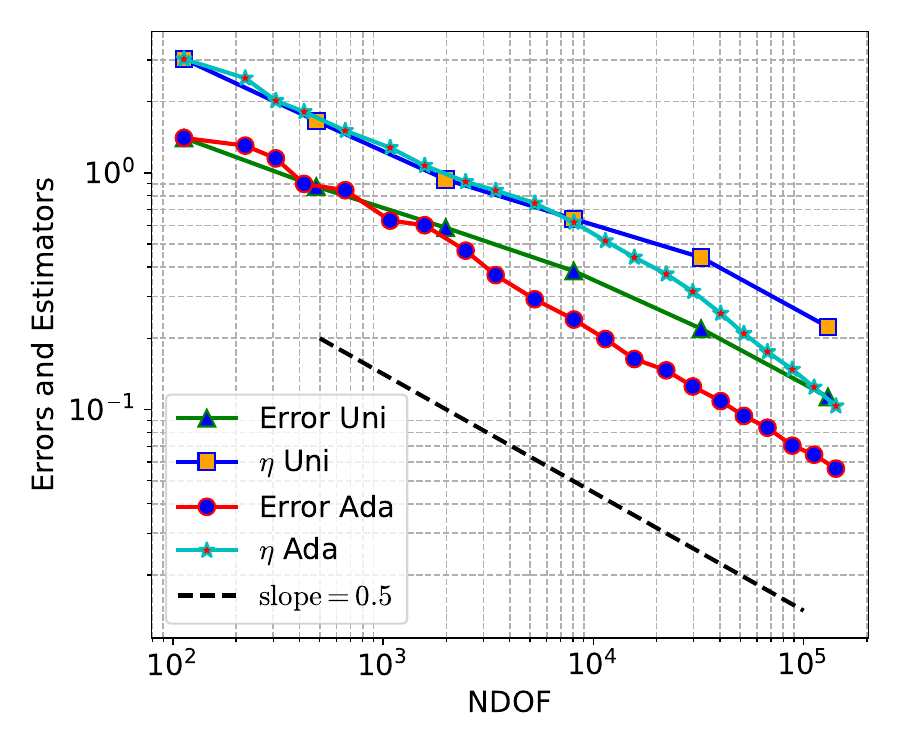}
\caption{Errors and estimators with $\varepsilon=10^{-2}$}
\label{f:xBL_fig_d}
\end{subfigure} 

\begin{subfigure}{.48\textwidth}
\includegraphics[width=\textwidth]{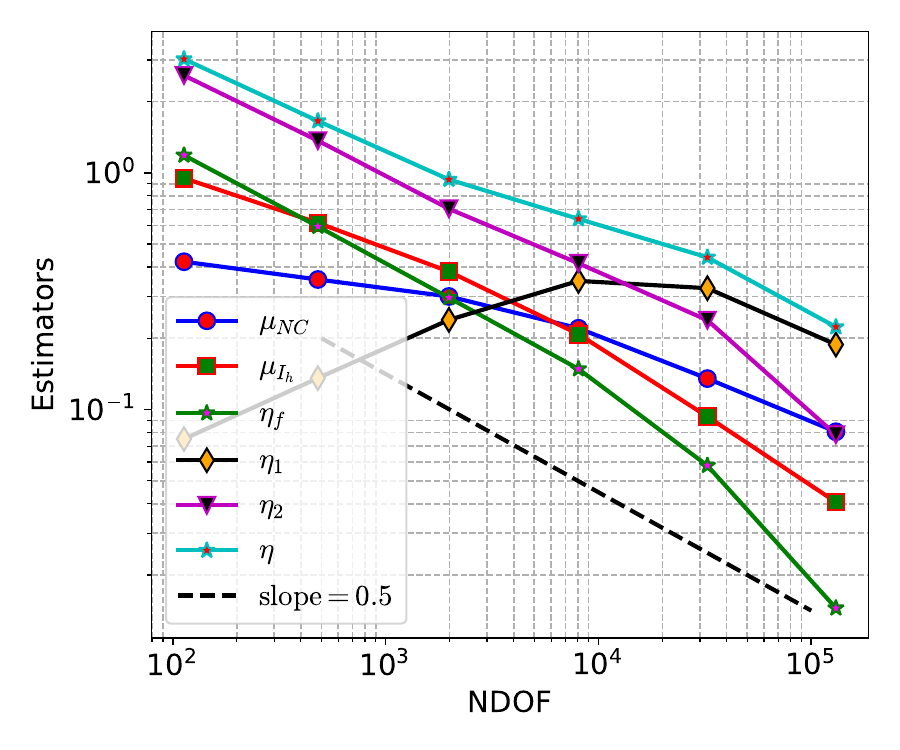}
\caption{Estimators on uniform meshes $(\varepsilon=10^{-2})$}
\label{f:xBL_fig_e}
\end{subfigure} 
\hfil
\begin{subfigure}{.48\textwidth}
\includegraphics[width=\textwidth]{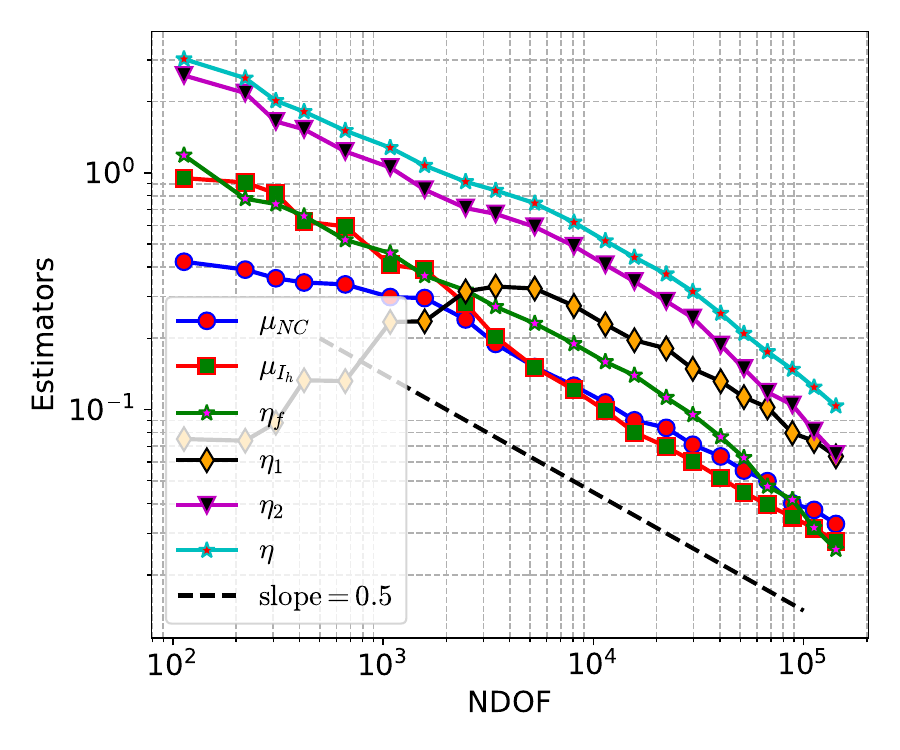}
\caption{Estimators on adaptive meshes $(\varepsilon=10^{-2})$}
\label{f:xBL_fig_f}
\end{subfigure} 

\caption{\small Adaptive meshes, errors, and estimators in Example \ref{eg_xBL}.
  ``Ada'' indicates adaptive, ``Uni'' indicates uniform mesh refinement.}
 %\label{xBL_fig}
\end{figure}

\begin{example}\label{eg_LBL}
Let $\Omega=(-1,1)^2\setminus ([0,1)\times (-1,0])$.
Consider the PDE \eqref{e:intro_sing_pert_strong_form}
with $f=(|x+y|)^{-1/3}$ \cite[Section 5.4]{GallistlTian2024}.
\end{example}
In this example, the exact solution is unknown.
The numerical experiment begins with an initial mesh that is obtained
by performing two successive red-refinements on the L-shaped domain,
which consists of six isosceles triangles. Figure~\ref{f:LBL_fig_a}
shows the total estimators $\eta$
for $\varepsilon\in\{1,10^{-2},10^{-5}\}$ under uniform and
adaptive mesh refinement.
Under uniform mesh refinement, the convergence rates of the
estimators are suboptimal due to the singularity at the
re-entrant corner.
Moreover, for $\varepsilon=10^{-2}$, a preasymptotic worse approximation
can be seen until approximately $3\cdot 10^{-3}$ degrees of freedom for
adaptive meshes, while the uniform meshes seem to suffer from
this effect in the whole range of performed computations.
For $\varepsilon=10^{-5}$, this preasymptotic effect is visible
both for uniform and adaptive meshes in the range of computations.
This effect is probably caused by the appearance of the boundary
layer.
While the adaptive mesh for $\varepsilon=1$ refines at the
re-entrant corner due to the singularity, see Figure~\ref{f:LBL_fig_b},
the adaptive mesh for $\varepsilon=10^{-2}$ resolves the singularity,
the boundary layer and the profile of $f$, see Figure~\ref{f:LBL_fig_c}.

\begin{figure}
\centering
\begin{subfigure}{.42\textwidth}
\includegraphics[width=\textwidth]{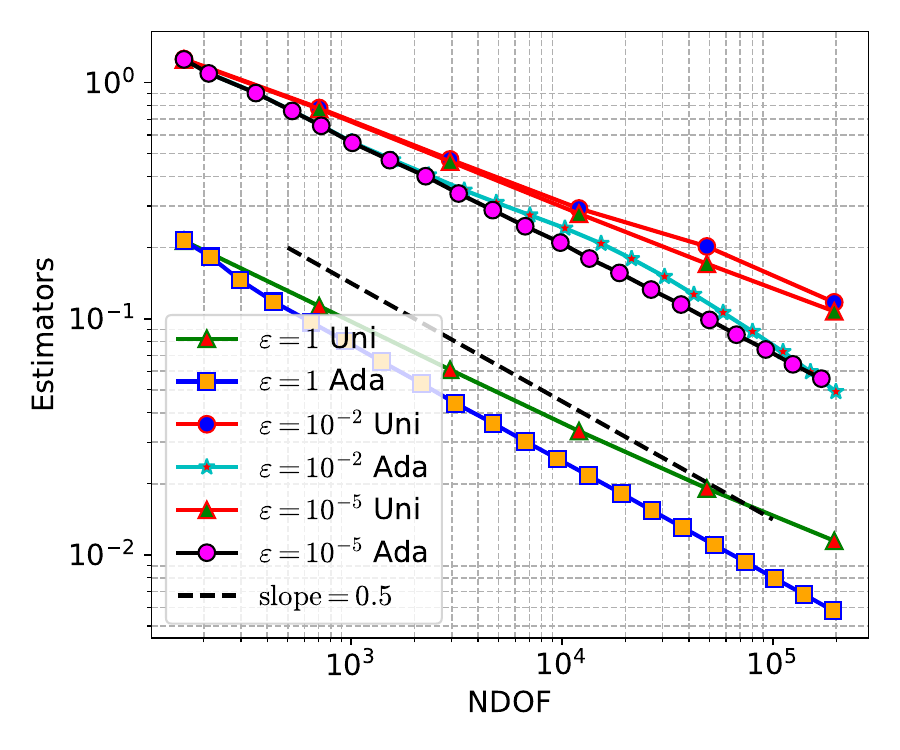}
\caption{Total estimators}
\label{f:LBL_fig_a}
\end{subfigure}

\begin{subfigure}{.42\textwidth}
\includegraphics[width=\textwidth]{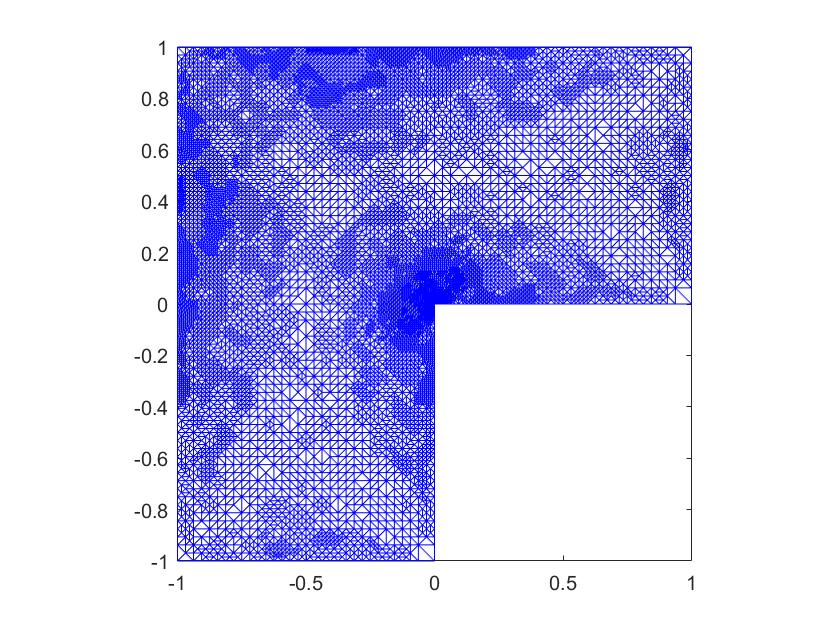}
\caption{Adaptive mesh with $\varepsilon=1$ (NDOF 37\,401)}
\label{f:LBL_fig_b}
\end{subfigure}
\hfil
\begin{subfigure}{.42\textwidth}
\includegraphics[width=\textwidth]{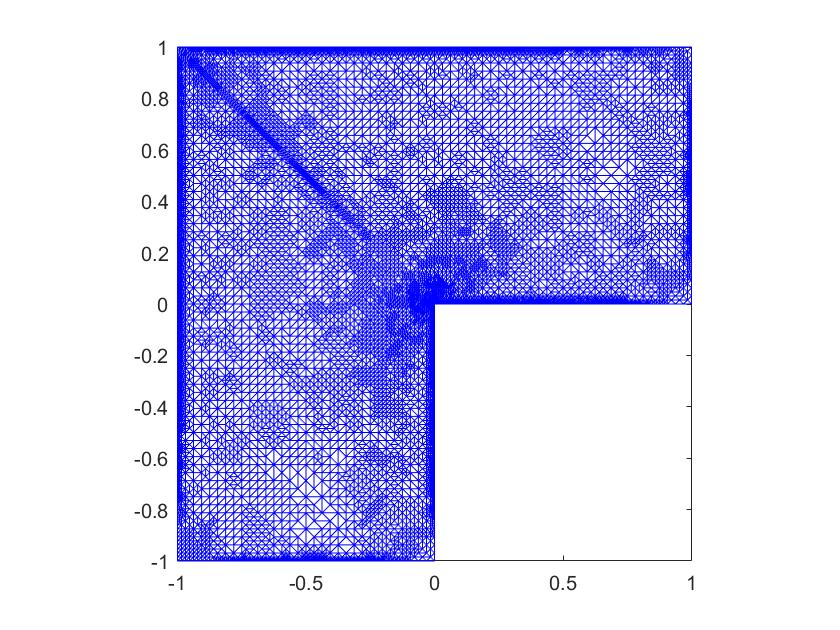}
\caption{Adaptive mesh with $\varepsilon=10^{-2}$ (NDOF 42\,231)}
\label{f:LBL_fig_c}
\end{subfigure}
\caption{\small Estimators and adaptive mesh with various values of $\varepsilon$ in Example \ref{eg_LBL}}
% \label{LBL_fig}
\end{figure}

\subsection{Von K{\'a}rm{\'a}n equation}
Two numerical experiments are considered for the von K{\'a}rm{\'a}n equation. 
The first is formulated on a convex domain, whereas the second is on a non-convex domain.
In the first case, the exact solution is smooth, while
in the second example, the exact solution exhibits a corner singularity.
 The performance of the errors and estimators under uniform and adaptive mesh refinements for both examples is presented below.

\begin{example}\label{eg_vke1}
    Consider the von K{\'a}rm{\'a}n equation on the unit square domain with the exact solution $\Psi=(\psi_1,\psi_2)$ where  ${\psi}_1=\sin^2(\pi x)\sin^2(\pi y)$ and ${\psi}_2=x^2 y^2 (1-x)^2 (1-y)^2$. 
\end{example}

\begin{example}\label{eg_vke2}
   Let $\Omega=(-1,1)^2\setminus \operatorname{conv}\{ (0,0), (1,-1/2),(1,0) \}$
   be a domain with cusp with an interior angle 
   $\omega=7\pi/4$.
   Consider the von K{\'a}rm{\'a}n equation with the exact solution
   \cite{Grisvard1992} given in polar coordinates by   
   $\Psi=(\psi_1,\psi_2)$ where ${\psi}_1={\psi}_2=(r^2\cos^2 \theta -1)^2(r^2\sin^2 \theta -1)^2r^{1+\gamma}g_{\gamma,\omega}(\theta)$ with $\gamma= 0.5006083\dots$ is a non-characteristic root of $\sin^2(\gamma \omega)=\gamma^2\sin^2(\omega)$,
   and $g_{\gamma,\omega}(\theta)=\big( \frac{1}{\gamma-1}\sin((\gamma-1)\omega)-\frac{1}{\gamma+1}\sin((\gamma+1)\omega)\big) (\cos((\gamma-1)\theta)-\cos((\gamma+1)\theta))-\big( \frac{1}{\gamma-1}\sin((\gamma-1)\theta)-\frac{1}{\gamma+1}\sin((\gamma+1)\theta)\big)(\cos((\gamma-1)\omega)-\cos((\gamma+1)\omega)$. 
\end{example}

Figures \ref{f:vke2_fig_a}--\ref{f:vke2_fig_b} show the errors
$\lVert D^2_{\mathrm{pw}}(\Psi-\Psi_h)\rVert_{L^2(\Omega)}$
($H^2$ err),
$\lVert \nabla(\Psi-I_h\Psi_h)\rVert_{L^2(\Omega)}$ ($H^1$ int err), and the complete estimator ($\eta$) under uniform (Uni) and adaptive (Ada) refinements for Examples \ref{eg_vke1} and \ref{eg_vke2}.

As displayed in Figure~\ref{f:vke2_fig_a},
all error quantities and the complete estimator for Example \ref{eg_vke1}
under both the uniform and adaptive mesh refinements, exhibit the optimal convergence rate of $0.5$ with respect to the number of degrees of freedom (NDOF), which is expected since the domain is convex. 
In contrast, for Example \ref{eg_vke2} the $H^2$-seminorm error shows suboptimal convergence under uniform refinement due to the presence of corner singularities, which motivates the use of adaptive mesh refinement. It is observed that under adaptive mesh refinement, both the errors and the estimator (see Figure \ref{f:vke2_fig_b}) achieve optimal convergence rates.
The error estimator contributions are separately plotted in Figure~\ref{f:vke2_fig_c}.
Furthermore, Figure~\ref{f:vke2_fig_d} illustrates the adaptive mesh at the 15th iteration,
where strong refinement is observed near the critical region.

\begin{figure}
\centering
\begin{subfigure}{.48\textwidth}
\includegraphics[width=\textwidth]{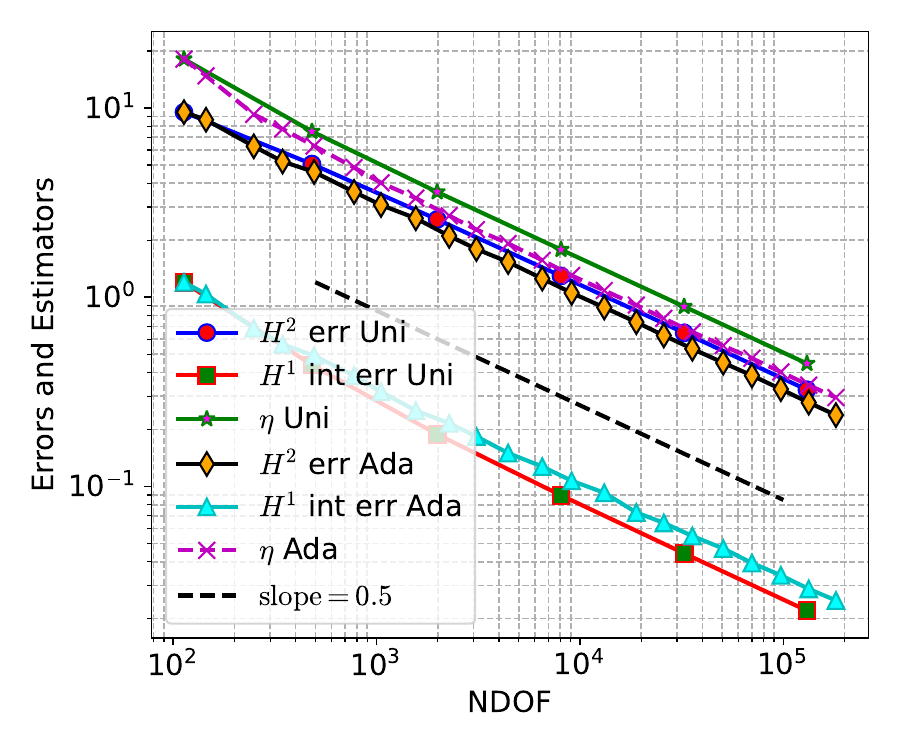}
\caption{Comparison of errors and estimators for Ex.~\ref{eg_vke1}}
\label{f:vke2_fig_a}
\end{subfigure}
\hfil 
\begin{subfigure}{.48\textwidth}
\includegraphics[width=\textwidth]{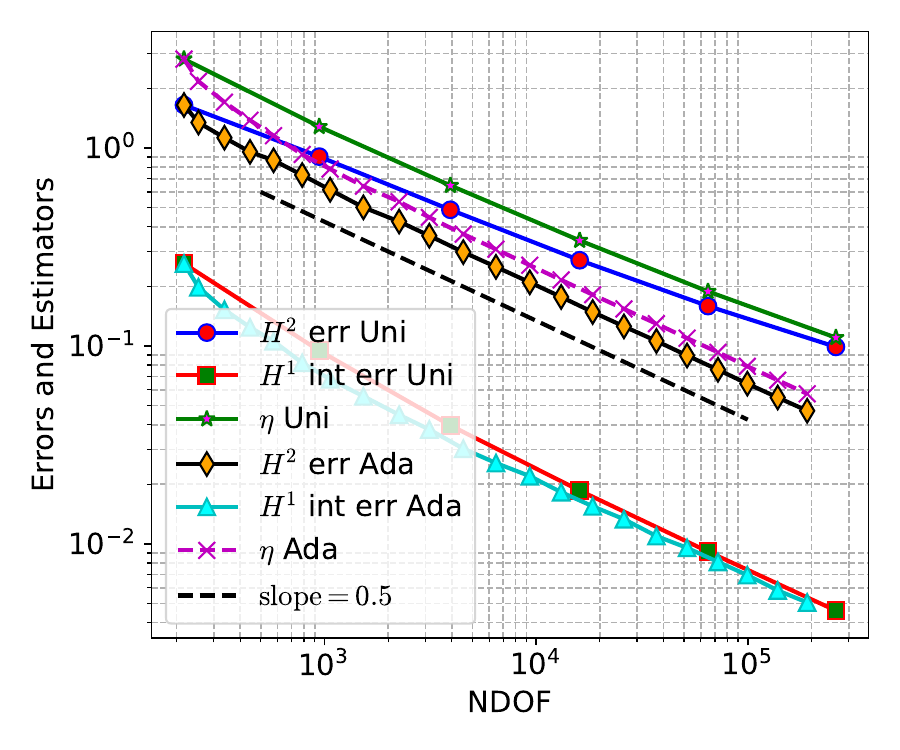}
\caption{Comparison of errors and estimators for Ex.~\ref{eg_vke2}}
\label{f:vke2_fig_b}
\end{subfigure}

\begin{subfigure}{.48\textwidth}
\includegraphics[width=\textwidth]{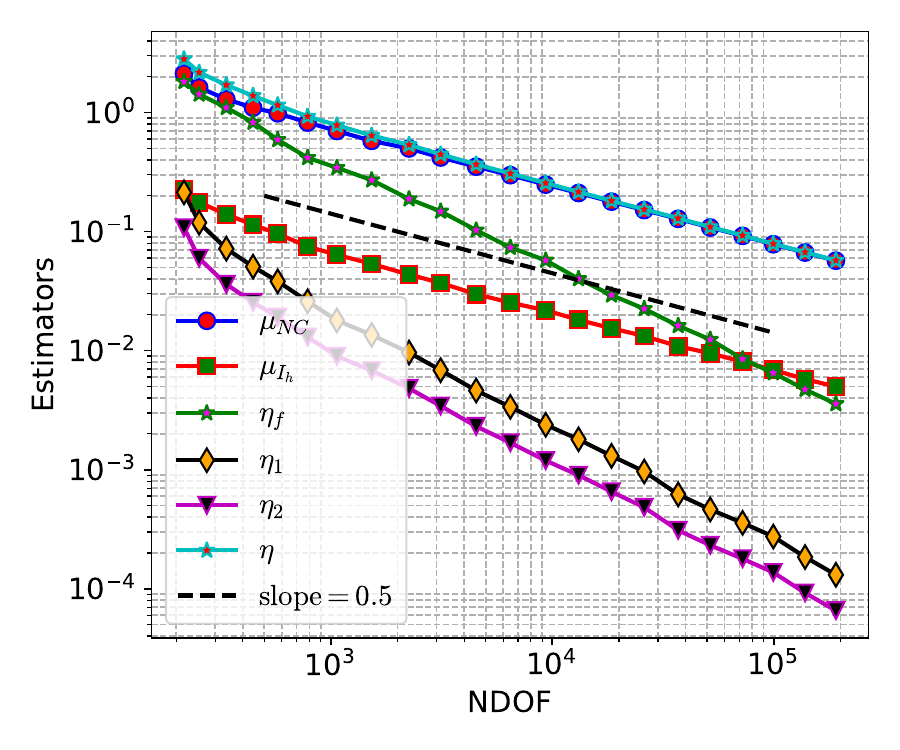}
\caption{Estimators on adaptive meshes in Ex.~\ref{eg_vke2}}
\label{f:vke2_fig_c}
\end{subfigure}
\hfil
\begin{subfigure}{.48\textwidth}
\includegraphics[width=\textwidth]{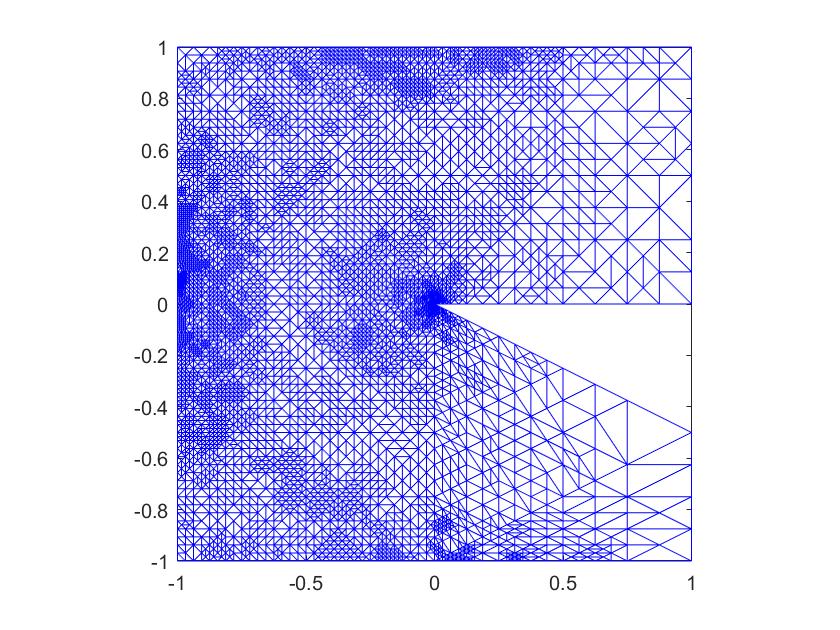}
\caption{Adaptive mesh (18\,507 NDOF) for Ex.~\ref{eg_vke2}}
\label{f:vke2_fig_d}
\end{subfigure}
 \caption{\small Comparison of errors 
 and the complete a posteriori estimator for Examples \ref{eg_vke1} and \ref{eg_vke2}.}
 %\label{vke2_fig}
\end{figure}

\appendix

\section{Outline of the proof of Theorem~\ref{t:vK_apriori}}

The proof follows as in \cite{MallikNataraj2016} with the 
following small modifications.

The first step consists in the proof of the 
discrete inf-sup condition 
\begin{align*}
    \sup_{\lVert D^2_\pw \Theta_h\rVert_{L^2(\Omega)}=1}
      \tilde{\mathcal{A}}_\nc(\Theta_h,\Phi_h)
    \gtrsim \lVert D^2_\pw \Phi_h\rVert_{L^2(\Omega)}
\end{align*}
for the perturbed linearized bilinear form
\begin{align*}
    &\tilde{\mathcal{A}}_\nc(\Theta_h,\Phi_h)\\
      &\qquad
      :=A_\pw(\Theta_h,\Phi_h) + B_\pw(I_\mathcal{M}\Psi,I_h\Theta_h,I_h\Phi_h)
         + B_\pw(\Theta_h,I_hI_\mathcal{M}\Psi,I_h\Phi_h),
\end{align*}
compare also with \cite[Lemma~4.1]{MallikNataraj2016} without 
the nodal interpolation operator.
This is a consequence of the fact that the perturbation 
with respect to the linearized bilinear form 
\begin{align*}
    \mathcal{A}_\nc(\Theta_h,\Phi_h)
      :=A_\pw(\Theta_h,\Phi_h) + B_\pw(\Psi,\Theta_h,\Phi_h)
         + B_\pw(\Theta_h,\Psi,\Phi_h).
\end{align*}
is only in the lower-order term $B_\pw$.

In the second step, it can be shown that the map 
$\mu:\mathcal{M}_0(\tri;\R^2)\to\mathcal{M}_0(\tri;\R^2)$
defined by 
\begin{align*}
    \tilde{\mathcal{A}}_\nc(\mu(\Theta_h),\Phi_h)
    &= F(\Phi_h) + B_\pw(I_\mathcal{M}\Psi,I_h\Theta_h,I_h\Phi_h)\\
    &\qquad\qquad 
      + B_\pw(\Theta_h,I_hI_\mathcal{M}\Psi,I_h\Phi_h)
      - B_\pw(\Theta_h,I_h\Theta_h,I_h\Phi_h)
\end{align*}
maps a ball around $I_\mathcal{M}\Psi$ into itself.
The proof follows the lines of \cite[Theorem~4.2]{MallikNataraj2016} 
employing additionally the property 
\eqref{e:nodal_interpolation_approx}.
Moreover, $\mu$ is (locally around $I_\mathcal{M}\Psi$)
a contraction, which implies the (locally unique)
solvability of the discrete problem.

\bibliographystyle{alpha}
\bibliography{apost_morley}

\begin{thebibliography}{CDNS22}

\bibitem[AF03]{AdamsFournier2003}
Robert~A. Adams and John J.~F. Fournier.
\newblock {\em Sobolev Spaces}, volume 140 of {\em Pure and Applied Mathematics
  (Amsterdam)}.
\newblock Elsevier/Academic Press, Amsterdam, second edition, 2003.

\bibitem[BGS10]{BrennerGudiSung2010}
Susanne~C. Brenner, Thirupathi Gudi, and Li-Yeng Sung.
\newblock An a posteriori error estimator for a quadratic {$C^0$}-interior
  penalty method for the biharmonic problem.
\newblock {\em IMA J. Numer. Anal.}, 30(3):777--798, 2010.

\bibitem[BNS07]{BeiraodaVeigaNiiranenStenberg2007}
L.~{Beir{\~a}o da Veiga}, J.~Niiranen, and R.~Stenberg.
\newblock A posteriori error estimates for the {Morley} plate bending element.
\newblock {\em Numer. Math.}, 106(2):165--179, 2007.

\bibitem[BR80]{BlumRannacher1980}
H.~Blum and R.~Rannacher.
\newblock On the boundary value problem of the biharmonic operator on domains
  with angular corners.
\newblock {\em Math. Methods Appl. Sci.}, 2(4):556--581, 1980.

\bibitem[Bre03]{Brenner2003}
Susanne~C. Brenner.
\newblock {P}oincar{\'e}--{F}riedrichs inequalities for piecewise {{\(H^{1}\)}}
  functions.
\newblock {\em SIAM J. Numer. Anal.}, 41(1):306--324, 2003.

\bibitem[BS08]{BrennerScott2008}
S.~C. Brenner and L.~R. Scott.
\newblock {\em The Mathematical Theory of Finite Element Methods}, volume~15 of
  {\em Texts in Applied Mathematics}.
\newblock Springer, New York, third edition, 2008.

\bibitem[CDNS22]{ChowdhuryDondNatarajShylaja2022}
Sudipto Chowdhury, Asha~K. Dond, Neela Nataraj, and Devika Shylaja.
\newblock \emph{A posteriori} error analysis for a distributed optimal control
  problem governed by the von {K{\'a}rm{\'a}n} equations.
\newblock {\em ESAIM, Math. Model. Numer. Anal.}, 56(5):1655--1686, 2022.

\bibitem[Cia78]{Ciarlet1978}
Philippe~G. Ciarlet.
\newblock {\em The Finite Element Method for Elliptic Problems}, volume~4 of
  {\em Studies in Mathematics and its Applications}.
\newblock North-Holland, Amsterdam, 1978.

\bibitem[Cia22]{Ciarlet_bookII_2022}
Philippe~G. Ciarlet.
\newblock {\em Mathematical elasticity. {Volume} {II}. {Theory} of plates},
  volume~85 of {\em Class. Appl. Math.}
\newblock Philadelphia, PA: Society for Industrial {and} Applied Mathematics
  (SIAM), reprint of the 1997 edition edition, 2022.

\bibitem[CMN20]{CarstensenMallikNataraj2020}
Carsten Carstensen, Gouranga Mallik, and Neela Nataraj.
\newblock Nonconforming finite element discretization for semilinear problems
  with trilinear nonlinearity.
\newblock {\em IMA J. Numer. Anal.}, 41(1):164--205, 2020.

\bibitem[DE12]{DiPietroErn2012}
{Daniele Antonio} {Di Pietro} and Alexandre Ern.
\newblock {\em Mathematical Aspects of Discontinuous {G}alerkin Methods},
  volume~69 of {\em Math\'ematiques \& Applications (Berlin)}.
\newblock Springer, Heidelberg, 2012.

\bibitem[DLZ22]{DuLinZhang2022}
Shaohong Du, Runchang Lin, and Zhimin Zhang.
\newblock Residual-based a posteriori error estimators for mixed finite element
  methods for fourth order elliptic singularly perturbed problems.
\newblock {\em J. Comput. Appl. Math.}, 412:16, 2022.

\bibitem[D{\"o}r96]{Doerfler1996}
Willy D{\"o}rfler.
\newblock A convergent adaptive algorithm for {P}oisson's equation.
\newblock {\em SIAM J. Numer. Anal.}, 33(3):1106--1124, 1996.

\bibitem[Gal15]{Gallistl2015}
Dietmar Gallistl.
\newblock Morley finite element method for the eigenvalues of the biharmonic
  operator.
\newblock {\em IMA J. Numer. Anal.}, 35(4):1779--1811, 2015.

\bibitem[Gri92]{Grisvard1992}
P.~Grisvard.
\newblock {\em Singularities in Boundary Value Problems}, volume~22 of {\em
  Recherches en Math\'ematiques Appliqu\'ees}.
\newblock Masson, Paris, 1992.

\bibitem[GT24]{GallistlTian2024}
D.~Gallistl and S.~Tian.
\newblock A posteriori error estimates for nonconforming discretizations of
  singularly perturbed biharmonic operators.
\newblock {\em SMAI J. Comput. Math.}, 10:355--372, 2024.

\bibitem[Gud10]{Gudi2010}
Thirupathi Gudi.
\newblock A new error analysis for discontinuous finite element methods for
  linear elliptic problems.
\newblock {\em Math. Comp.}, 79(272):2169--2189, 2010.

\bibitem[HS09]{HuShi2009}
Jun Hu and Zhongci Shi.
\newblock A new a posteriori error estimate for the {Morley} element.
\newblock {\em Numer. Math.}, 112(1):25--40, 2009.

\bibitem[HSX12]{HuShiXu2012}
Jun Hu, Zhongci Shi, and Jinchao Xu.
\newblock Convergence and optimality of the adaptive {M}orley element method.
\newblock {\em Numer. Math.}, 121(4):731--752, 2012.

\bibitem[Kel75]{Keller1975}
H.~B. Keller.
\newblock Approximation methods for nonlinear problems with application to
  two-point boundary value problems.
\newblock {\em Math. Comput.}, 29:464--474, 1975.

\bibitem[Kni67]{Knightly1967}
G.~H. Knightly.
\newblock An existence theorem for the von {K{\'a}rm{\'a}n} equations.
\newblock {\em Arch. Ration. Mech. Anal.}, 27:233--242, 1967.

\bibitem[MN16]{MallikNataraj2016}
Gouranga Mallik and Neela Nataraj.
\newblock A nonconforming finite element approximation for the von {Karman}
  equations.
\newblock {\em ESAIM, Math. Model. Numer. Anal.}, 50(2):433--454, 2016.

\bibitem[Mor68]{Morley1968}
L.S.D. Morley.
\newblock The triangular equilibrium element in the solution of plate bending
  problems.
\newblock {\em Aeronaut.Quart.}, 19:149--169, 1968.

\bibitem[NTW01]{NilssenTaiWinther2001}
Trygve~K. Nilssen, Xue-Cheng Tai, and Ragnar Winther.
\newblock A robust nonconforming {{\(H^2\)}}-element.
\newblock {\em Math. Comput.}, 70(234):489--505, 2001.

\bibitem[Osw94]{Oswald1994}
Peter Oswald.
\newblock {\em Multilevel finite element approximation}.
\newblock Teubner Skripten zur Numerik. B. G. Teubner, Stuttgart, 1994.

\bibitem[Ver98]{Verfuerth1998}
R.~Verf{\"u}rth.
\newblock Robust a posteriori error estimators for a singularly perturbed
  reaction-diffusion equation.
\newblock {\em Numer. Math.}, 78(3):479--493, 1998.

\bibitem[Ver13]{Verfuerth2013}
R{\"u}diger Verf{\"u}rth.
\newblock {\em A posteriori error estimation techniques for finite element
  methods}.
\newblock Numer. Math. Sci. Comput. Oxford: Oxford University Press, 2013.

\bibitem[WM07]{WangMeng2007}
Ming Wang and Xiangrui Meng.
\newblock A robust finite element method for a 3-{D} elliptic singular
  perturbation problem.
\newblock {\em J. Comput. Math.}, 25(6):631--644, 2007.

\bibitem[WXH06]{WangXuHu2006}
Ming Wang, Jin-chao Xu, and Yu-cheng Hu.
\newblock Modified {M}orley element method for a fourth order elliptic singular
  perturbation problem.
\newblock {\em J. Comput. Math.}, 24(2):113--120, 2006.

\bibitem[ZW08]{ZhangWang2008}
Shuo Zhang and Ming Wang.
\newblock A posteriori estimator of nonconforming finite element method for
  fourth order elliptic perturbation problems.
\newblock {\em J. Comput. Math.}, 26(4):554--577, 2008.

\end{thebibliography}

\end{document}